\documentclass[manuscript,screen]{acmart}
\usepackage{booktabs} 
\usepackage{algorithm}
\usepackage{algorithmicx}  
\usepackage{algpseudocode}
\usepackage{amsfonts}
\usepackage{graphicx}
\usepackage{subfigure}
\usepackage{multirow}
\usepackage{amsmath}
\AtBeginDocument{%
  \providecommand\BibTeX{{%
    \normalfont B\kern-0.5em{\scshape i\kern-0.25em b}\kern-0.8em\TeX}}}

\setcopyright{acmcopyright}
\copyrightyear{2018}
\acmYear{2018}
\acmDOI{10.1145/1122445.1122456}

\acmConference[Woodstock '18]{Woodstock '18: ACM Symposium on Neural
  Gaze Detection}{June 03--05, 2018}{Woodstock, NY}
\acmBooktitle{Woodstock '18: ACM Symposium on Neural Gaze Detection,
  June 03--05, 2018, Woodstock, NY}
\acmPrice{15.00}
\acmISBN{978-1-4503-XXXX-X/18/06}

\begin{document}

\title{Optimized Multivariate Polynomial Determinant on GPU}

\author{Jianjun Wei}
\orcid{1234-5678-9012}
\affiliation{%
  \institution{East China Normal University}
}
\email{wjjfast@163.com}

\author{Liangyu Chen}
\affiliation{%
  \institution{East China Normal University}
}
\email{lychen@sei.ecnu.edu.cn}

%
%
%

\renewcommand{\shortauthors}{Jianjun Wei and Liangyu Chen.}

\begin{abstract}
We present an optimized algorithm calculating determinant for multivariate polynomial matrix on GPU. The novel algorithm provides precise determinant for input multivariate polynomial matrix in controllable time. Our approach is based on modular methods and split into Fast Fourier Transformation, Condensation method and Chinese Remainder Theorem where each algorithm is paralleled on GPU. The experiment results show that our parallel method owns substantial speedups compared to Maple, allowing memory overhead and time expedition in steady increment. We are also able to deal with complex matrix which is over the threshold on Maple and constrained on CPU. In addition, calculation during the process could be recovered without losing accuracy at any point regardless of disruptions. Furthermore, we propose a time prediction for calculation of polynomial determinant according to some basic matrix attributes and we solve an open problem relating to harmonic elimination equations on the basis of our GPU implementation.
\end{abstract}



\maketitle

\section{Introduction}
Determinant is a value that can be computed from the elements of a square matrix and encodes certain properties of the linear transformation described by the matrix. It is an important and fundamental algorithm with a wide range of application from symbolic computing to numeric algorithm, such as resultants of two polynomials and Vandermonde identity, or solving equations with high degree.

The resultant of two univariate polynomials over a field or over a commutative ring is commonly defined as the determinant of their Sylvester matrix\cite{Villard:2018:CRG:3208976.3209020}. This allows the process to judge whether two polynomials have relatively prime or multiple roots merely according to the value of the determinant, avoiding complex process to get the solution of the equations. Vandermonde identity also benefits from the determinant when being used in polynomial interpolation since inverting the Vandermonde matrix allows expressing the coefficients of the polynomial in terms of the $\alpha_i$ and the values of the polynomial at the $\alpha_i$\cite{Floater:2017:PII:3137843.3137972}. Also, Vandermonde determinant is useful in the representation theory of the symmetric group\cite{fox2018irreducible}.

However, using CPU computing determinant is a daunting task and yields major issues during the calculation. For numeric matrix, although $2 \times 2$ determinant could be calculated easily, computing determinant is more time-consuming for large matrices owe to colossal of elementary row or column operations\cite{faal2018new}. For symbolic matrix, the standard way to calculate determinant is to break it down into more determinants of lower degree by taking the products of any row or column entry and the determinant of its complementary minor, then alternately adding and subtracting the results. Yet, this process brings expansion of the minors during the process and memory proliferates fast, increasing the complexity. 

Direct application of determinant on CPU also lacks the amount of parallelism. A prevailing computing method for determinant is on a single CPU which makes possible the serial computing that a problem is broken into a discrete series of instructions that are executed one after another, but it dismisses as the performance of single CPU the complexity of determinant and the power of multi-core and many-core architecture. Even though problems can be run with multiple CPUs so that they can be solved simultaneously, the number of CPU is confined. 

This makes using Graphics processing units (GPU) desirable because of its sheer memory bandwidth compared to CPU\cite{Mittal:2015:SCH:2775083.2788396}. The GPU is becoming very attractive for scientific computing applications that are targeting high performance. And, a large portion of thread scheduling is done through the decomposition of the application into serious of GPU kernels. Rather, their highly parallel structure renders them more efficient than general-purpose CPUs for algorithms that process large blocks of data in parallel \cite{DBLP:conf/parco/HaqueMX15}.

A GPU implementation of determinant can be used to accelerate CPU by transferring matrix into GPU memory, calculating determinant on GPU, and copying back the value to CPU memory. With CUDA, a parallel computing platform and programming model, developers are able to dramatically speed up computing applications by harnessing the power of GPUs\cite{Kirk:2007:NCS:1296907.1296909}. However, in the general case of symbolic matrices coefficients of the results are usually compressed or replaced by approximate values so that they can be presented and calculated on CPU. This also makes parallelization a challenge. Maintaining accuracy during process requires equivalent transformation from real number field into finite field with modular methods\cite{DBLP:conf/issac/DahanMSWX05}. In this paper, we present algorithms that calculate determinant for symbolic matrix over big prime field on GPU. Our proposal is to get precise value of the symbolic determinant in predictable time. We focus on multivariate polynomial matrix and our algorithm relies on modular methods.


For our multivariate polynomial determinant algorithm over big prime field, we transfer symbolic computation to numerical computation with modular method since it is not suitable, even inefficient, to present polynomial expression in memory, which would occupy large bulk of limited space. Our parallel algorithm could be split into three major parts: fast Fourier transform (FFT), condensation numeric determinant method and Chinese Remainder Theorem (CRT). We optimize these algorithms so that evaluation and interpolation could be used on multivariate polynomials. In addition, the performance of condensation numeric determinant method and CRT has been improved by reducing row or column matrix operations and amplifying GPU attributes. We report a GPU implementation and our experimental results show that:

\begin{itemize}
  \item [(1)] 
  For multivariate polynomial matrix with much higher order or degree, our GPU algorithm could provide precise value of its determinant in acceptable time and obtain significant speedups with respect to the software calculating on CPU.
  \item [(2)]
  We provide an algorithm that will not suffer from drastic memory soaring, which means our algorithm retains its computational ability and offers refined calculation when encountering matrix that could not afford by CPU.
   
  \item[(3)]
  We could forecast the running time according to some of the polynomial matrix attributes. Besides, our algorithm is able to go on calculating at any point even though it is disrupted by some other factors.
\end{itemize}

The rest of the paper is organized as follows: Section 2 presents the whole algorithm and explains how symbolic computing could be transferred into numeric computing. In Section 3, optimized algorithms are described as well as their GPU implementations to show how our determinant on GPU could be used to accelerate determinant on CPU. Section 4 shows the experimental results. In Section 5, we offer a time estimation for the computation. Section 6 explains an open problem relating to harmonic elimination equations that could not be calculated on CPU and show how it is solved using the GPU method we proposed. Finally, the conclusions are stated. 

\section{determinant algorithm}
As indicated above, implementation of symbolic determinant on CPU will occupy large amounts of memory, thus, we transfer symbolic computation into numerical computation. However, numerical computation still has some unsolved problems, especially with regard to its computational availability such as big coefficients and precision. To figure out an optimal solution, we consider the best method to do numerical computation in our algorithm is the modular method over prime field\cite{article:ACTRC}, and the reasons are as follows:


\begin{itemize}
  \item [(1)]
    The numerical expression, such as coefficients of variable, could be presented and calculated on CPU and GPU.
  \item [(2)] 
    The modular method assures that numerical matrix has the same meaning as symbolic matrix but differs just in expression. Therefore, accuracy would not be lost during the process.

\end{itemize}

\begin{algorithm}
\caption{Determinant algorithm using modular method} 
\label{alg:determinant} 
\begin{algorithmic}[1] 
\Procedure{Determinant}{$vn, r, \vec{d},\vec{p}$}
\State compute $\vec{p}$ size: $pn$,
\State let $i = 0$, 
\While{$i < pn$}
\State FFT$(vn, r, \vec{d}, p_i)$,
\State DET$(r, p_i)$,
\State IFFT$(vn, r, \vec{d}, p_i)$,
\EndWhile
\label{euclidendwhile} 
\State D $:=$ CRT $(\vec{d}, \vec{p}, pn)$,
\State \textbf{return} $(D)$
\EndProcedure 
\end{algorithmic} 
\end{algorithm}

Algorithm \ref{alg:determinant} computes the determinant using certain modular method over prime field, which is shown in Figure \ref{flow_path} . We represent each chosen prime that we use in modular method as a vector $\vec{p} = (p_0, p_1, \dots, p_{pn})$ of length $pn$, and the degree of each variable as a vector $\vec{d} = (d_0, d_1, \dots, d_{vn} )$ of length $vn$. 

Now, given a symbolic matrix with order $r$, we explain how to choose the prime as a unique member of prime field. According to the coefficients of each element in the matrix, the upper bound of the determinant coefficient could be easily computed and we note it $boundary$. Then, for $i = 0, \dots, k$ we shall have $ p_0 \cdot p_1 \cdots p_k >= boundary$, which means the product of the prime in prime field should be over the upper bound where $k$ is the minimum that satisfy the formula. Meanwhile, we enlarge the degree of each variable into at most a power of 2, and note the largest one among them $2^{q_{max}}$. This gives a $\vec{d}$ of $(15, 7, 9)$, and we would replace it with $(16, 8, 16)$. We observe that $ q_{max} =4$ since the largest degree is 16. The prime in the prime field should admits $2^{q_{max}}$-th primitive root, thus prime $p$ could be represented as:

\begin{equation}
\label{primes}
  p = c \cdot 2^{q_{max}} + 1
\end{equation}
where $c$ is a constant number. Usually, we choose large prime with 9 digits since starting with smaller primes such as 97 results in more primes that are required to comprise the equivalent prime field, which is time-consuming when reconstructing large coefficients.


\noindent \textbf{FFT}. We present our implementation for constructing FFT-based evaluation method that transfers the symbolic computation into numerical computation. The first step is to evaluate the element of input symbolic matrix, the polynomial more precisely. For our implementation, polynomial is encoded as a vector containing all of its coefficients by expanding the amount of its items into power of 2 by filling 0. At this point, the polynomial can be cut evenly when its expanded version is too large to putting into global memory on GPU or not fitting threads grids well. For example, $F = 1 + 2x + 3xy + 4x^2 + 5y^2$ can be expanded as $F = 1 + 0 \cdot y + 5y^2 + 2x + 3xy + 0 \cdot xy^2 + 4x^2 + 0 \cdot x^2y + 0 \cdot x^2y^2$ and encoded as a vector [1, 0, 5, 2, 3, 0, 4, 0, 0].


\begin{figure*}[htbp]
\includegraphics[height=2.75in, width=5.9in]{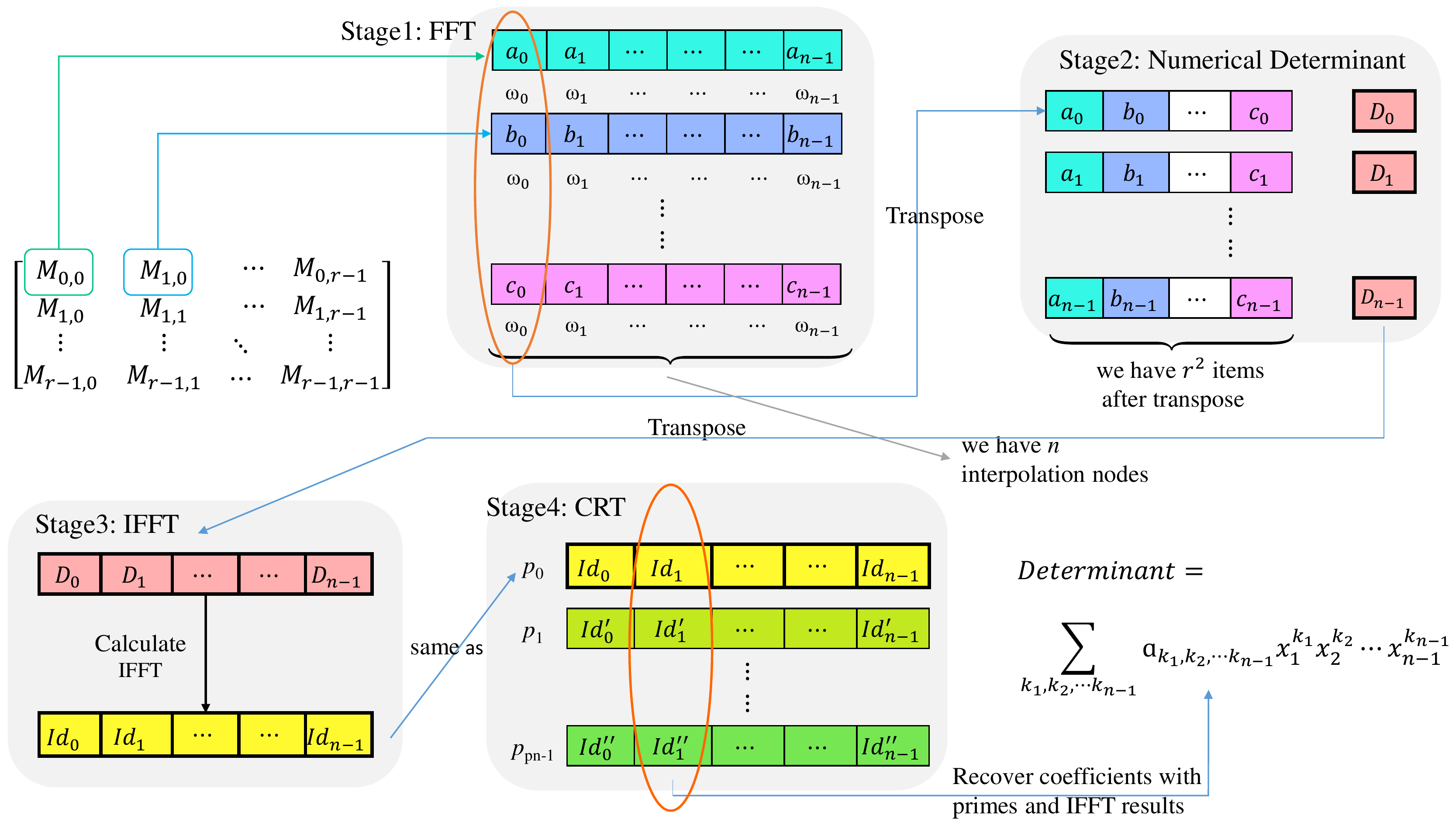}
\caption{4-stage optimized determinant method. Black halos represent interim results, which are shifted as a whole.}
\label{flow_path}
\end{figure*}

Only when the degree boundary, that is $d_0 \cdot d_1 \cdots d_{vn}$ noted as $n$, exactly equals the number of interpolation nodes could we definitely construct a polynomial, thus, we need $n$ interpolation nodes during our method. An $N$-th root of unity, where $k$ is a positive integer for $ 1 < k \leq N$, is a number $\omega \in \mathcal{R}$ satisfying the equation $\omega^k = 1$. The element $\omega \in \mathcal{R}$ is a $principle$ $N$-th root of unity if $\omega^N = 1$ and for all $ 1 < k \leq N$, we have

\begin{equation}
  \sum_{j=0}^{N-1} \omega^{jk} = 0.
\end{equation}
Recall the Discrete Fourier Transform(DFT) over finite field\cite{EPSFFTGI, mohajerani2016fast}. Let $\omega \in \mathcal{R}$ be a principle $N$-th root of unity. The $N$-point DFT at $\omega$ is a linear function, mapping the vector $\vec{a}=(a_0, \dots,a_{N-1})^T$ to $\vec{b}=(b_0, \dots,b_{N-1})^T$ by $\vec{b}=\Omega \vec{a}$, where $\Omega = (\omega^{jk})_{0 \leq j,k \leq N-1}$. Assume that $N$ could be factorized $JK$, with $J,K > 1$. Recall Cooley-Tukey factorization formula:

\begin{equation}
\label{cooley-turkey}
  DFT_{JK} = (DFT_J \otimes I_K)D_{J,K}(I_J \otimes DFT_K)L_{J}^{JK}.
\end{equation}

For two matrices $A$, $B$ over $\mathcal{R}$ with respective formats $m \times n$ and $q \times s$, we denote $A \otimes B$ over $\mathcal{R}$ an $mq \times ns$ matrix over $\mathcal{R}$ called the tensor product of A by B and defined by 

\begin{equation}
\label{AtimesB}
  A \otimes B = \lbrack a_{kl}B \rbrack_{k,l} \quad with \quad  A = \lbrack a_{kl}  \rbrack _{k,l}.
\end{equation}

In Equation \ref{AtimesB}, DFT$_{JK}$, DFT$_{J}$ and DFT$_{K}$ are respectively the $N$-point DFT at $\omega$, the $J$-point DFT at $\omega^K$, the $K$-point DFT at $\omega^J$. The $stride \ permutation$ matrix $L_{J}^{JK}$ permutes an input vector $x$ of length $JK$ as follows

\begin{equation}
   x \lbrack iJ+j \rbrack  \mapsto x\lbrack iJ+j \rbrack ,
\end{equation}
for all $0 \leq j < J$, $0 \leq i < K$. $L_{J}^{JK}$ performs as a transposition of a matrix if $x$ is viewed as a $K \times J$ matrix. The $ diagonal \ twiddle \ matrix \ D_{J,K}$ is defined as:

\begin{equation}
   D_{J,K} = \bigoplus_{j=0}^{J-1}diag(1, \omega^j, \dots, \omega^{j(K-1)} ).
\end{equation}
Equation \ref{cooley-turkey} implies various divide-and-conquer algorithms for computing DFTs efficiently, often referred to fast Fourier transforms\cite{DBLP:conf/issac/ChenCMM17}. One of the famous variant of Cooley-Tukey FFT is Stockham FFT, which could be defined as:

\begin{equation}
\label{Stockham}
   DFT_{2^l}= \prod_{i=0}^{l-1}(DFT_2 \otimes I_{2^{l-1}})(D_{2,2^{l-i-1}} \otimes I_{2^i})(L_2^{l-i} \otimes I_{2^i}),
\end{equation}
where $I_t$ is the identity matrix of order $t$. Stockham FFT could be found in \cite{DBLP:conf/afips/Stockham66}. For the Stockham factorization, the identity matrix $I_t$ only appears on the right while Cooley-Tukey FFT \cite{belilita2017enhanced} on the left. The key difference brings a significant performance gap since Cooley-Tukey FFT is accessing the powers of $\omega$ by performing larger and larger jumps by computing not only the powers $(1, \omega, \dots, \omega^{n/2-1})$ but also all jumped power while Stockham FFT avoids the problem by packing all the accesses to a power of $\omega$ together, resulting in a broadcasting inside a thread block. Thus, Stockham FFT achieves coalesced accesses, which could be combined as one instruction by GPU, making parallel computing less a problem. Hence, we optimize the Stockham FFT so that it can be performed on multivariate polynomial and implement it on GPU.


\noindent \textbf{DET}. According to the result of FFT, each element of the matrix, which is actually a polynomial, is transferred into a particular numerical value at the interpolation node which is noted $\vec{\omega_i} =(\omega_0, \omega_1, \dots, \omega_{vn})$, where each $\omega_i$ corresponds to variable. Thus, for any $i = 0, \dots, n-1 $, each $\vec{\omega }_i $ matches a numerical matrix $M$, that is $ \vec{\omega }_i  \mapsto M(\vec{\omega }_i)$. We compute the determinant of $M$ and note it $D$. It appears that $ \vec{\omega }_i  \mapsto D(\vec{\omega }_i)$, which could also be represented as ordered pair like $(\vec{\omega }_i, D(\vec{\omega }_i))$. Condensation method is used to calculate the numerical determinant but its limitation lies on the account of elementary transformation in matrix. We will optimize the condensation method on GPU so that it increases throughput and the performance is improved.

\noindent \textbf{IFFT}. Recall that $(\vec{\omega }_i, D(\vec{\omega }_i))$ consists of interpolation node set. Since each $\vec{\omega }_i$ is different from others, there is a unique polynomial $A$ that $D(\vec{\omega }_k) = A(\vec{\omega }_k)$ for $k = 0, 1, \dots, n-1$, thus, $A$ is the determinant on a certain prime. IFFT is the inverse process of FFT and its implementation on GPU is also similar to the one of FFT.

\noindent \textbf{CRT}. The Chinese remainder theorem is widely used to compute large integers, designed to replace the original computation with several smaller ones. We observe that field $\mathbb{Z}/p\mathbb{Z}$ could be divided into several subfields such as $\mathbb{Z}/p\mathbb{Z} \simeq \mathbb{Z}/p_0\mathbb{Z} \times \cdots \times \mathbb{Z}/p_{pn-1} \mathbb{Z}$. CRT provides a much faster way than the direct computation if the prime $p$ and the number of operations are large, which recovers the determinant of symbolic matrix without sacrificing computation accuracy. Since each prime $p_i$ in the prime field $\mathbb{Z}/p_i\mathbb{Z}$ maps a particular polynomial representing the determinant over $\mathbb{Z}/p_i\mathbb{Z}$ called $A_i$, we consider the problem of finding the precise determinant $D$ of the input symbolic matrix over field $\mathbb{Z}$ as:

\begin{center}
$D \equiv A_0 \mod p_0$

$\vdots$

$D \equiv A_{n-1} \mod p_{pn-1}$.

\end{center}
The CUDA implementation on GPU is presented in Section 3. The parallel computing divides the large problem into smaller ones and carries out simultaneously, leading a high performance computing.

\section{IMPLEMENTATION ON GPU}
Section 2 suggests strategies that are useful to avoid soaring memory and computation complexity. However, the improved algorithms suffer from the following bottlenecks, if implemented iteratively. In one respect, CPU is burdened by serious atomic memory contention. On the other hand, the runtime imposes a maximum limit when coping with great deal of numerical computing in series algorithm that is detailed above. Nevertheless, these aforementioned limitations can be alleviated by using parallel computing, which is more effective than general-purpose CPU on processing algorithms with large block of data. This entails GPU because of its highly parallel structure.

In this section, we discuss the GPU implementation techniques over $\mathbb{Z}/p\mathbb{Z}$ field. We begin with the optimized FFT algorithm, and then proceed with an in-depth discussion of the condensation algorithm which is used to compute numerical determinant. IFFT is used to produce accurate determinant of original matrix $M$ over $\mathbb{Z}/p\mathbb{Z}$ field. The final step is the Chinese reminder theory, which is used to recover the exact value of the determinant of input polynomial matrix.

\subsection{FFT in  $\mathbb{Z}/p\mathbb{Z}$ }
As we have pointed above, Stockham FFT is an extraordinary method which is geared up for GPU when transferring symbolic expression into numerical expression in the first stage.  
Thus we suggest parallelization strategies for Stockham FFT, which could be applied not only for univariate polynomials but also for multivariate polynomials. Here we illustrate the implementation of Stockham FFT based on GPU, employing the mentioned strategies purposed in Section 2. Recall that Stockham FFT could be divided into three steps:

\begin{itemize}
  \item [(1)] 
    $S_1: x \longrightarrow (L_2^{l-i} \otimes I_{2^i})x$.
  \item [(2)]
    $S_2: x \longrightarrow (D_{2,2^{l-i-1}} \otimes I_{2^i})x$.
  \item [(3)]
    $S_3: x \longrightarrow (DFT_2 \otimes I_{2^{l-1}})x$.
\end{itemize}

According to the definition, $L_2^{l-i} \otimes I_{2^i}$ realize this stride permutation of the matrix $M$ whose effect is to perform the reordering as matrix transposition. $D_{2,2^{l-i-1}} \otimes I_{2^i}$ is a diagonal matrix of size $2^{l-i}$ and thus $D_{2,2^{l-i-1}} \otimes I_{2^i}$ is again a diagonal matrix of size $n$, with each diagonal element repeated $2^i$ times. Hence, step $S_2$ simply scales $x$ with powers of the primitive root of unity $\omega$. Step $S_3$ is a list of basic butterflies with the stride size $n/2$, and this step accesses data in a very uniform manner.

\subsubsection{Multivariate FFT on GPU}

Solving univariate polynomial system, which is mentioned in \cite{2012JPhCS.385a2014A}, and bivariate polynomial system in \cite{DBLP:journals/cca/MazaP11} still have some limitations, let alone multivariate polynomial system for symbolic computation. However, our optimized efforts on multivariate FFT algorithm can be oriented to overcome these limitations.  

Let $F$ be the vector that consists of the coefficients of the polynomial, which is one of the element of matrix $M$. In each iteration of the algorithm, we update the coefficients according to each variable and its primitive root $\omega$. After $vn$ iterations, all variables have been calculated, and thus, one multivariate polynomial has been evaluated completely, which is the result of Stockham FFT on the given polynomial. The pseudocode is given in Algorithm \ref{alg:FFT_GPU}.

\begin{algorithm}[htb]  
\caption{Stockham FFT algorithm optimized on GPU} 
\label{alg:FFT_GPU} 
\begin{algorithmic}[1] 
\Procedure{StockhamFFTOnGPU}{$vn, r, \vec{d}, p_i, F$}
\For{$j=0$ to $vn-1$}
\State let $colDim = d_{vn-1}$, $rowDim = d_0d_1 \cdots d_{vn-2}$ 
\State We compute how much data needed to be sent to GPU and noted $perSendRow$
\For{$i=0$ to $rowDim$ \textbf{by} $perSendRow$}
\State FFT\_GPU$(F, vn, r, \vec{d}, p_i)$
\EndFor
\State TranspositionOnGPU$(F,colDim, rowDim)$
\State let $colDim = d_{vn-1}$
\For{$i=vn-2$ to $0$}
\State $d_{i+1} = d_i$
\EndFor
\State $d_0 =  colDim$

\EndFor

\EndProcedure 
\end{algorithmic} 
\end{algorithm}

The algorithm above is split in three parts: line 3 where one variable will be chosen to be evaluated; lines 4-7 where how much data will be decided to send and evaluate in GPU; and line 8-13 where the position of the variable will be exchanged in $\vec{d}$. In what follows, we will refer to these parts respectively. For the first part, we choose the last variable $d_{vn-1}$ as "unknown" variable, which is waiting for evaluating and the product of the rest as "known" variable, which means that these variables have already been evaluated. In detail, we optimize the Stockham FFT for multivariate polynomial as follows:

\begin{itemize}
  \item [(1)] 
    For univariate polynomial, we evaluate the polynomial directly.
  \item [(2)]
    For bivariate polynomial, writing the variable as the ordered pair $G(x_2, x_1)$ for convenience where the first position of $G$ is regarded as "unknown" while the second position is "known", we first assume $x_1$ as "known" and evaluate $x_2$ with interpolation nodes which is composed by the corresponding primitive root $\omega$. Then, we evaluate $x_1$ using similar way. In this algorithm, we consider the polynomial, or its coefficient expression $F$ more precisely, as a rectangle. We do the transposition after evaluating $x_2$ for the reason that coalesced memory on GPU could be accessed. The process is shown in Figure \ref{bivariate FFT polynomial}.
  \item [(3)]
    For multivariate polynomial, that is the polynomial whose variable amount is larger than 3, we evaluate the polynomial with the method similar to the one in processing bivariate polynomial with the difference that the data will be seen as a multidimensional vector. We consider the multidimensional vector as two-dimensional vector by choosing one variable as "unknown" and the rest as "known". Concretely, let variables as an ordered pair $H(x_1, x_2, \dots, x_n)$, it is obvious that $H$ could be transferred into as $G^\prime(x_n, x_1 x_2 \cdots x_{n-1})$ with two dimension where we evaluate $x_n$ while the rest is considered "known". Thus multivariate polynomial has been transferred into bivariate polynomial to a certain degree. Iteratively, the ordered pair becomes $G^{\prime\prime}(x_n x_{n-1}, x_1 x_2 \cdots x_{n-2})$ when evaluating $x_{n-1}$. The process is shown in Figure \ref{multivariate FFT polynomial}.
\end{itemize}

\begin{figure}[htbp]
\centering

\subfigure[Evaluate bivariate polynomial.]{
\begin{minipage}[t]{1\linewidth}
\label{bivariate FFT polynomial}
\centering
\includegraphics[width=2.9in]{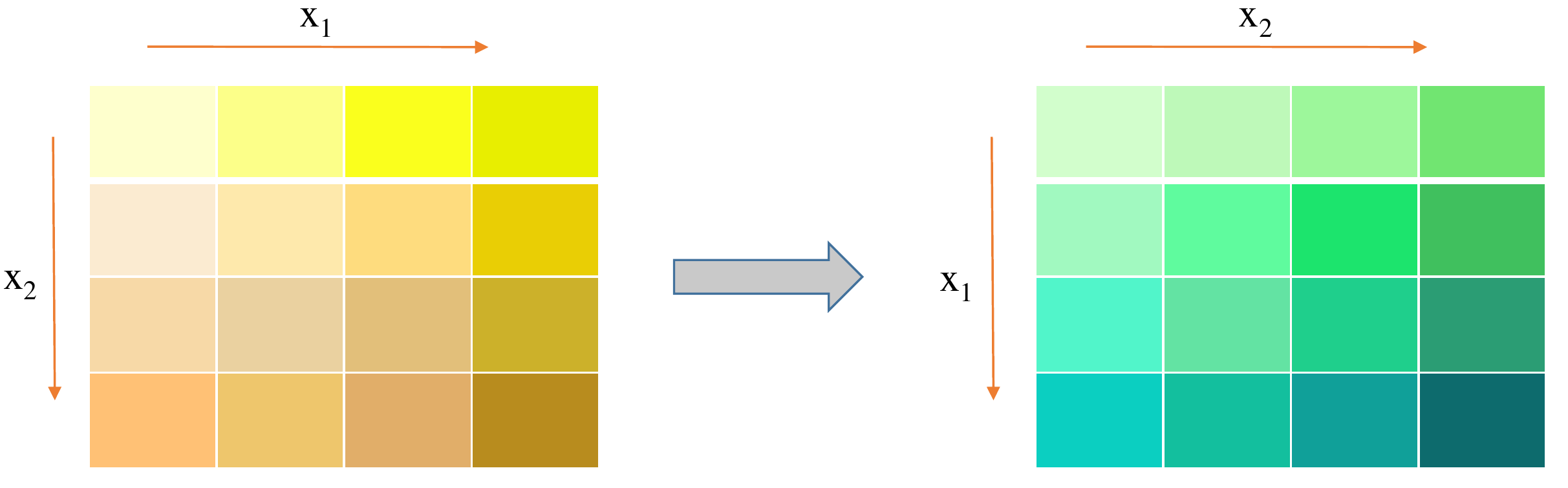}
\end{minipage}%
}%

\subfigure[Evaluate multivariate polynomial, here $v_n=3$.]{
\begin{minipage}[t]{1\linewidth}
\label{multivariate FFT polynomial}
\centering
\includegraphics[width=3.0in]{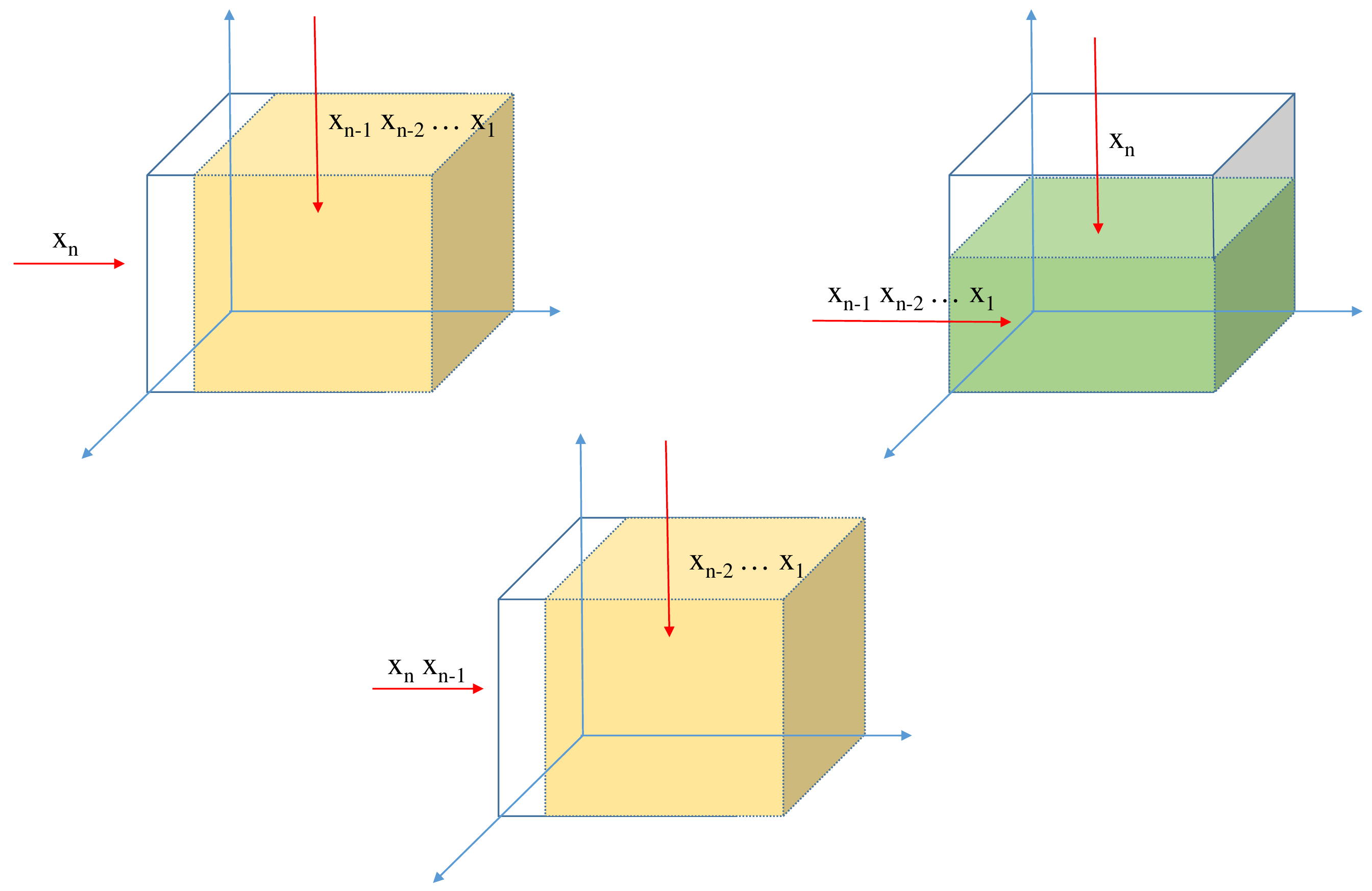}
\end{minipage}
}%

\centering
\caption{Multivariate Stockham FFT on GPU.}
\end{figure}

Transposition \cite{DBLP:conf/ppopp/SungGGGH14} is done after one variable has been evaluated on GPU. In this part, shared memory is fully used so that the bank conflicts and lock conflicts \cite{RGMP:optimizing} could be effectively avoided. 

As what is shown in third part of the algorithm, we move the last variable to the first position in $\vec{d}$ after evaluation and transposition, for instance, $(x_1, x_2, x_3)$ would be transformed into $(x_3, x_1, x_2)$, then $(x_2, x_3, x_1)$. This assures that all variables have been evaluated.

\subsubsection{Implementation of Stockham FFT on GPU}
The second part of the algorithm evaluates the polynomial on GPU based on the optimized algorithm above. At first, we compute the amount of data we will send to GPU according to $rowDim$ and $BDIM$, the amount of blocks enumerated and distributed to multiprocessors in a kernel grid. We note these data $perSendRow$. After that, we decide how many blocks and threads we should designate to the kernel on the basis of $perSendRow$. Here we set the maximum of $colDim$ as 8192 and note it $MAX\_FFT\_DIM$ since it is such a large degree for one variable that it could cover most of the situations in polynomial system and satisfy the demands of the most practical applications. When processing polynomial matrix on GPU, we could conclude three situations:

\begin{itemize}
  \item [(1)] 
    When $perSendRow$ is less than $TDIM$, the amount of threads in one block, we designate one thread block to evaluate the polynomial since threads in per block are sufficient to deal with input polynomial even one thread processes one coefficient. Typically, the number of threads depends on the degree of input polynomials. 
  \item [(2)]
    Under such circumstance that one block could not afford the input coefficients of the polynomial, which means $perSendRow$ is larger than $TDIM$, we designate $block\_num$ thread blocks to compute FFT where $colDim$ items could be evaluated in one block. $block\_num$ is given in Equation \ref{block_num}. As discussed in detail in multiprocessor level, the fewer registers a kernel uses, the more threads and thread blocks are likely to reside on a multiprocessor, which can improve performance. However, assigning too much active blocks may also do bad to occupancy. Thus, we calculate $\left \lfloor{MAX\_FFT\_DIM / colDim} \right \rfloor \times colDim$ items in one block where resources on GPU could be fully utilized.

\begin{equation}
\label{block_num}
   block\_num =   \left \lfloor \frac {perSendRow} {MAX\_FFT\_DIM / colDim} \right \rfloor
\end{equation}

  \item [(3)]
    When $colDim$ is so large that GPU could not offer computation even though all blocks are working, we still designate $block\_num$ blocks to kernel and the number of threads per block is selected as $TDIM$. However, all available blocks would be used more than once until $perSendRow$ items are evaluated. In other words, blocks would be used iteratively with several rounds. 
\end{itemize}

Recall that Stockham FFT is divided into three steps, which fully maps quit straightforward to the GPU. We discuss the realization of these kernels below, which is shown in Figure \ref{SituationsOfFFT}. Each step will experience these three situations.

\begin{figure}
\includegraphics[height=2.35in, width=3.85in]{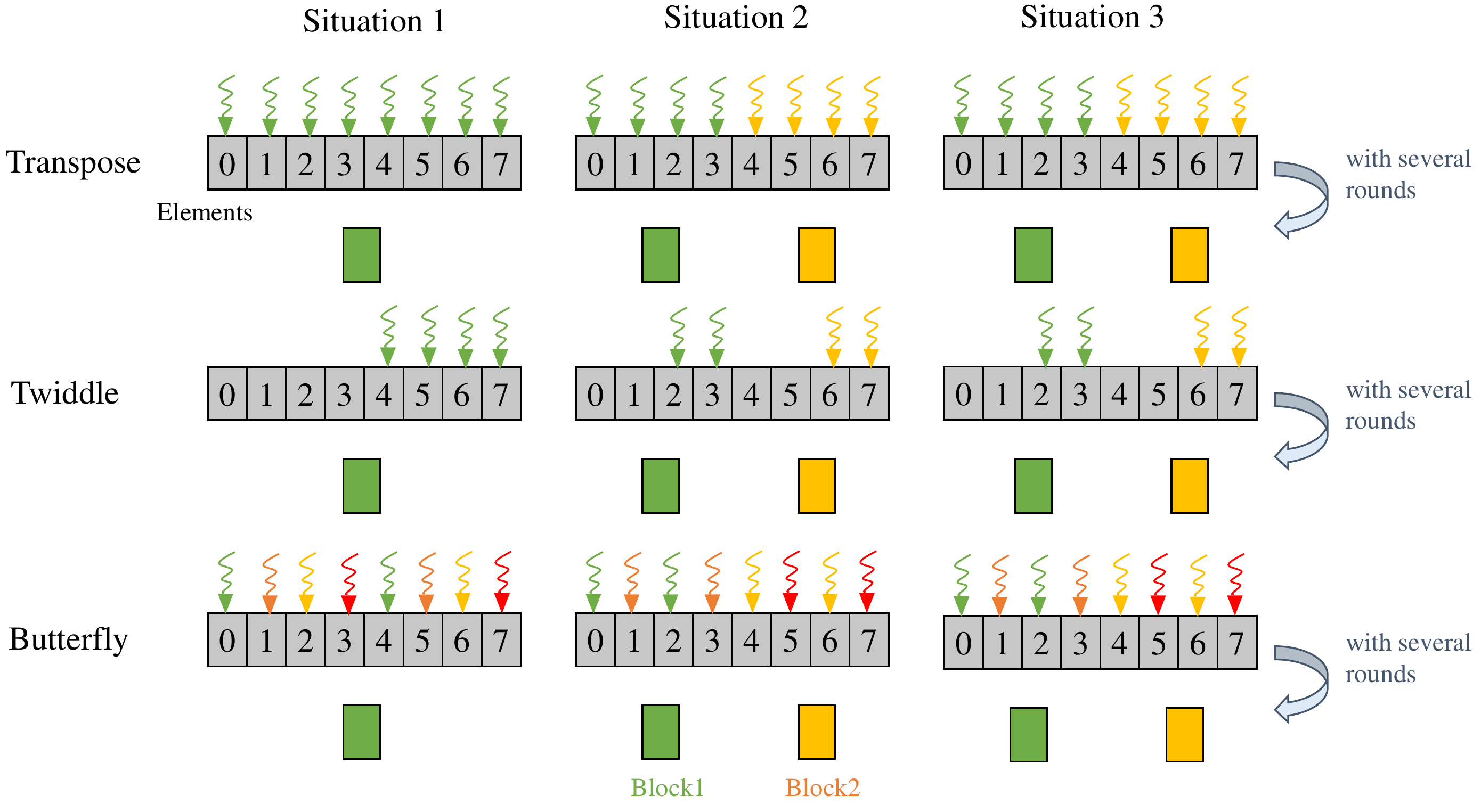}
\caption{Situations when processing Stockham FFT on GPU. In Situation 1, one block processes all elements. In Situation 2, one block processes half of the elements. In Situation 3, two blocks are used with several rounds.}
\label{SituationsOfFFT}
\end{figure}

\noindent \textbf{Stride transpose of Stockham FFT}. Transposition in Stockham FFT is similar to the one in \cite{RGMP:optimizing}. For each situation above, we use one thread process one coefficient. At first, we transfer all data into shared memory so that non-coalesced memory access such as that in global memory could be avoid. Then, we compute the position of the matrix and write the number at this position to the corresponding position of output matrix directly\cite{DBLP:conf/ppopp/CatanzaroKG14}.

\noindent \textbf{Stride twiddle of Stockham FFT} Twiddle is just a multiplication between coefficient of input polynomial and its corresponding evaluation point, which is based on $\omega$. Since only half of the data elements need to be processed in twiddle, the same amount of threads in one block now have the ability to deal with twice as much data as that in the Stride transpose step above. Thread workload is perfectly balanced because each thread does the same job. Finally, a compacted sequence is written back to global memory, which is provided for the butterflies in the next step.

\noindent \textbf{Butterfly of Stockham FFT}. The standard way of computing a butterfly is to alternately add and subtract the results calculated by twiddle. Since Stockham FFT divides the input polynomial into two parts, we associate one thread with two data elements. By assigning threads as this way, we ensure that all threads are occupied in the procedure. We keep the modular and corresponding reciprocals (inv) needed for modular computation in constant memory space because one thread block uses a single modulus throughout all computations. Accordingly, all threads of a block read the same value and the data is loaded via constant memory cache. Besides, direct reference to the data in constant memory has a positive effect on reducing register pressure.

To sum up, we would like to add that all steps of our evaluation algorithm discussed above execute in $O(\log n)$  parallel time on the GPU with $n$ threads, thereby making the final complexity also in linear parallel time.

%


\subsection{Numerical determinant in  $\mathbb{Z}/p\mathbb{Z}$}
In what follows, we will consider implementation of GPU kernel evaluating numerical determinants on each interpolation node, which is shown in Algorithm \ref{alg:condensation_GPU}. 

The procedure consists of two parts with the difference that in lines 4-7 we use one unique thread to find the first number from left to right unequal to 0 as well as its index in the first "uncalculated" row while in lines 9-12 the rest rows of matrix would be updated according to the number we record simultaneously. Theoretically, we could record first non-zero number corresponding to one specific row, and if not, the determinant is 0. For "uncalculated", we mean the given row that we have not recorded its non-zero number and its index.

\begin{algorithm}[htb]  
\caption{Condensation method optimized on GPU} 
\label{alg:condensation_GPU} 
\begin{algorithmic}[1] 
\Procedure{DeterminantOnGPU}{$r, p_i$}
\State let $tid = threadIdx.x + blockIdx.x \times blockDim.x$,
\For{$i=0$ to $r-1$}
\If{(tid \% r) == 0}
\State we find first number that is not 0, then store 
\State  its value as $z_j$ and its index $Id_j$ into shared
\State  memory respectively,
\EndIf
\State \_\_syncthreads,
\If{(tid \% r) > i}
\State update number before $Id_j$ of each row,
\State set $Id_j$ of each row as 0,
\State update number after $Id_j$ of each row,
\EndIf
\EndFor
\State we compute the product of $D_z=z_0z_1\cdots z_{r-1}$,
\State \_\_syncthreads,
\EndProcedure 
\end{algorithmic} 
\end{algorithm}

Between first part and second part is a garden variety thread barrier which is called synchronization where any thread reaching the barrier waits until all of the other threads in that block also reach it. The first non-zero number $z_j$ and its index $Id_j$ of each "uncalculated" row of respective matrix is stored in the shared memory so that all the threads in the block can have access to them. For the second part, we have three steps:

\begin{itemize}
  \item [(1)] 
    First, we update number before index $Id_j$. For $i=0$ to $Id_j$ and $j=0$ to $r-1$, we set $M_{j,i}$ as $mul\_mod(-M_{j,i}, n_j, p)$ where $mul\_mod$ is the computation of product of two digits $x,y$ in prime field $\mathbb{Z}/p\mathbb{Z}$.
  \item [(2)]
    Then, we note the value of $M_{j,Id_j}$ as $t$ before setting $M_{j,Id_j}$ as zero.
  \item [(3)]
    In the end, we update number after index $Id_j$. Similar with $mul\_mod$, $sub\_mod$ is the computation of subtraction of two digits $x,y$, meaning $x-y$ in prime field $\mathbb{Z}/p\mathbb{Z}$. For $Id_j$ to $r$, we calculate $s_1$ as $mul\_mod(z_j, M_{j,i}, p)$ and $s_2$ as $mul\_mod(M_{j,i}, t, p)$, and then, we set $T_i$ as $sub\_mod(s_1, s_2, p)$. 
\end{itemize}

Virtually, the steps above have the same function as matrix elementary transformation but its frequency plunges under our method. Furthermore, it is parallel and completely in-place in our implementation. This process sets all the numbers that below the non-zero number with the same index $Id_j$ as 0 in the meantime the remaining numbers in matrix updated respectively. Here one thread now processes one "uncalculated" row at a time, thus numbers such process involves are updated simultaneously. With more rows become "calculated", the matrix are completely updated and it could be transformed into an upper triangular matrix according to elementary transformation rules though there is no need to do so. Therefore, the determinant $D$ is the product of all non-zero number in $\mathbb{Z}/p\mathbb{Z}$: $z_0 z_1 \cdots z_{r-1} \mod p$.

Observe that, the number of working threads decreases during the process. Therefore, we use a load balancing strategy to improve thread occupancy: when at least half the threads enter the idle state, we switch to another code subroutine where computations are organized in a way that threads only do half a job. Eventually, once the size of numbers which waits to be updated descends below the warp boundary, the remaining algorithm steps are run without thread synchronization because warp, as a minimal scheduling entity, always executes synchronously on the GPU. 

The high memory bandwidth of GPUs makes them attractive to accelerate determinant computation. In detail, the peak for PCI-ex16 is 4GB/s versus 80GB/s peak for Quadro FX5600\cite{DBLP:conf/ppopp/SungGGGH14}. Hence, we transfer as much data as possible from CPU to GPU according to the size of the global memory. It should also be taken into account that more active threads should be put on a Streaming Multiprocessor(SM) to get a comparatively higher occupancy. Yet bear in mind that higher occupancy does not always equate to higher performance while low occupancy always interferes with the ability to hide memory latency, resulting in performance degradation. Thus, we evenly distribute data so that each SM could become load balanced and relatively more blocks can be used. 

Besides, the efficiency of this algorithm largely depends on how good the distribution of threads in per block is. Performing fast numerical determinant computation is not an easy task to accomplish because the graphics hardware is heavily optimized for coalesced memory access while reading and writing may be divided into several instructions by compiler due to misaligned or non-coalesced memory access. Divided into 3 steps when dealing with one unique row of matrix on GPU makes it harder to have access to coalesced memory, which directly contributes to lower load efficiency. Nevertheless, the high throughout and occupancy compensates the defect and the determinant could be calculated just in a few microseconds.

\subsection{IFFT in  $\mathbb{Z}/p\mathbb{Z}$}
Recall that interpolation nodes $(\vec{\omega }_i, D(\vec{\omega }_i))$ are generated in numerical determinant computation above. IFFT is used to calculate the interpolating polynomial on these data points in $\mathbb{Z}/p\mathbb{Z}$. Generally, there is exactly one polynomial due to the uniqueness of the polynomial interpolation which is proven in theorem \ref{uniquess of polynomial}.  

\begin{theorem}
\label{uniquess of polynomial}
Given $n$ points ${(x_0, y_0), (x_1, y_1), \dots , (x_{n-1}, y_{n-1})}$ with unequal $x_k$, there is at most one polynomial $A$ of degree less or equal to $n-1$ such that $y_k = A(x_k), k = 0, 1, \dots, n-1$.

\end{theorem}

\begin{proof}[\textbf{Proof}]
As we all know, a polynomial can be presented as:

\begin{equation}
\left[
\begin{array}{ccccc}
1 & x_0 &  x_0^2    & \cdots & x_0^{n-1}      \\
1 & x_1 &  x_1^2    & \cdots & x_1^{n-1}      \\
\vdots &  \vdots & \vdots & \ddots & \vdots \\
1 & x_{n-1} &  x_{n-1}^2    & \cdots & x_{n-1}^{n-1}      
\end{array}
\right ]
\left[
\begin{array}{cccc}
 a_{0}\\
 a_{1}\\
 \vdots \\
 a_{n}
\end{array}
\right ]
=
\left[
\begin{array}{cccc}
 y_{0}\\
 y_{1}\\
 \vdots \\
 y_{n}
\end{array}
\right ]
\end{equation}
where the left matrix could be represented as $V(x_0, x_1, \dots, x_n)$ called $Vandermonde$ matrix and its determinant is:

\begin{equation}
\prod_{0<j<k<n-1} (x_k-x_j)
\end{equation}

According to the theorem that a matrix is singular if and only if its determinant is 0, we can conclude that the matrix is invertible if $x_k$ is unequal. Therefore, given $(x_k, y_k)$, we can exactly generate one polynomial with coefficients $\vec{\alpha }$ :

\begin{equation}
\vec{\alpha } = V(x_0, x_1, \dots, x_n)^{-1}\vec{y }
\end{equation}

\end{proof}

We omit the implementation of IFFT since it closely resembles that of the FFT on GPU considered above but with the opposite sign in the exponent and a $1/N$ factor. At this point,  the generated IFFT  result is a unique polynomial in $\mathbb{Z}/p\mathbb{Z}$, which is also the exact determinant of the original polynomial matrix $M$ but in $\mathbb{Z}/p\mathbb{Z}$.

\subsection{Chinese reminder theory in  $\mathbb{Z}/p\mathbb{Z}$}
IFFT provides the precise determinant $Q(p_i)$ over $\mathbb{Z}/p\mathbb{Z}$ with respective prime $p_i$ for $i=0,1 \dots, n-1$. Since $p_i$ are pairwise coprime, it is nature to use Chinese remainder theorem to reconstruct the integer coefficient of the determinant over $\mathbb{Z}$. However, compared with conventional Chinese remainder theorem, Mixed-Radix conversion(MRC) algorithm has the property that computation could be finished without resorting to multi-precision arithmetic, which become decisive for choosing this algorithm in the realization on the GPU instead of conventional method \cite{DBLP:journals/jpdc/Emeliyanenko13}.

More precisely, for a given set of primes $(p_0, p_1, \dots, p_{n-1})$ and respective residues $(x_0, x_1, \dots, x_{n-1})$, which is composed by the coefficients of $Q(p_i)$ with the same term, the MRC algorithm \cite{MMRCRNSAA}associates the large integer $X$ with mixed-radix(MR) digits ${\alpha_i}$ in the following way:

\begin{equation}
\label{eq:X}
 X=\alpha_0 m_0 + \alpha_1 m_1 + \cdots + \alpha_{n-1} m_{n-1}
\end{equation}
where $m_0 = 1, m_j=p_0 p_1 \cdots p_{n-1}(j=1,\dots,n-1)$. With set of digits ${\alpha_i}$, large integer $X$ can also be easily obtained by simply evaluating a Horner's rule:

\begin{equation}
\label{eq:horner}
 X=\alpha_0 + m_0 ( \alpha_1 + m_1 (\alpha_2 + m_2 (\cdots + \alpha_{n-1})\cdots)).
\end{equation}
For $i = 0,1,\dots,n-1$, $\alpha_i$ could be evaluated as follows:

$\alpha_0 = x_0, \alpha_1 = (x_1 - \alpha_0) c_1 \bmod p_1,$

$\alpha_2 = ((x_2-\alpha_0) c_2 - (\alpha_1 m_1 c_2 \bmod p_2)) \bmod p_2, \dots$

$\alpha_i = ((x_i-\alpha_0) c_i - (\alpha_1 m_1 c_i \bmod p_i) - \cdots - (\alpha_{i - 1} m_{i-1} c_i)) \bmod p_i $
where $c_i = (p_0 p_1 \cdots p_{i-1})^{-1} \bmod p_i$.

\begin{algorithm}[htb]  
\caption{Chinese remainder theorem on GPU} 
\label{alg:CRT_GPU} 
\begin{algorithmic}[1] 
\Procedure{CrtOnGPU}{$\vec{p}, pn, \vec{xs}, \vec{as}, \vec{m}, \vec{c}$}
\State let $tid = threadIdx.x + blockIdx.x \times blockDim.x$
\State $x = xs + tid \times rPitch$
\State $a = as + tid \times aPitch$
\State store $\vec{m}, \vec{c}$ into shared memory as $\vec{ms}, \vec{cs}$
\State let $a_0 = x_0$
\For{$i=1$ to $pn$}
\State $t1 = sub\_mod(x_i,a_0,p_i)$
\State $t2 = mul\_mod(t1,cs_i,p_i)$
\State $t3 = 0$
\For{$j=1$ to $i-1$}
\State $t4 = mul\_mod(a_j, ms_j, p_i)$
\State $t4 = mul\_mod(t4, cs_i, p_i)$
\State $t3 = add\_mod(t3, t4, p_i)$
\EndFor

\State  $a_i = sub\_mod(t2, t3, p_i)$
\EndFor

\EndProcedure 
\end{algorithmic} 
\end{algorithm}

Although it is hard to be paralleled since the calculation of ${\alpha_i}$ strongly depends on the calculation of ${\alpha_j}$ for $j < i$, we still could conceive a simple parallel algorithm updating ${\alpha_{i+1}}$ because there are still large number of similar computations such as $\alpha_{i - 1} m_{i-1} c_i$. 

We have also observed that the number of residues $x_i$ and primes $p_i$ is far less than the number of coefficient in $D_s$, the determinant of the original symbolic matrix. Hence, parallel computing on the coefficients in $D_s$, the MR digits of each large integer coefficients more precisely, is more attractive and effective than parallel computing on a single large integer coefficient. 

Our GPU implementation is shown in Algorithm \ref{alg:CRT_GPU}. $\vec{xs}$ regards as residues, which is the coefficients of $Q(p_i)$ over $\mathbb{Z}/p\mathbb{Z}$. $\vec{m}, \vec{c}$ regards as $m$ and $c$ in Equation \ref{eq:X} and Equation \ref{eq:horner} respectively. We use serial method calculate one coefficient and parallel one computing the rest. Line 6-14 explicitly states the expression in Equation \ref{eq:horner} where one thread is used to calculate one large integer coefficient, in other words, the ${\alpha_i}$ of each large integer coefficient is calculated simultaneously. Also, we precompute the $\vec{m}$ and $\vec{c}$ in advance and transfer all of them into global memory for the moment. In the GPU kernel, $\vec{m}$ and $\vec{c}$ will be stored in the shared memory, a small on-chip memory with low latency so that all threads in one block share this address space and the data could be reused with fast access. It is unnecessary and not wise to put residues $\vec{xs}$ on shared memory since its storage on chip is limited, and thus, might not be large enough to keep all residues. Meanwhile, the more the amount of shared memory is used, the less the number of blocks become active, which will likely lead to low occupancy. To display large integer coefficients from mixed-radix digits on the host machine, we have employed the GMP 6.1.2 library \footnote{http://gmplib.org.}.

\section{Experiment}
Experiments in this section have been performed on one current NVIDIA device, using integer version of the algorithms. NVIDIA GeForce 2080 Ti with Turing architecture has a peak bandwidth 616GB/s, with 68 streaming multiprocessor and the total number of 4352 CUDA cores. First, we focus on the running time and memory overhead during the process, and then, we discuss the feasibility of our algorithm by analyzing available primes we could use.

\begin{figure*}[htbp]
\centering
\subfigure[Running time on GPU versus Maple.]{
\begin{minipage}[t]{0.5\linewidth}
\label{running_time_1}
\centering
\includegraphics[width=3.0in , height= 2.0in]{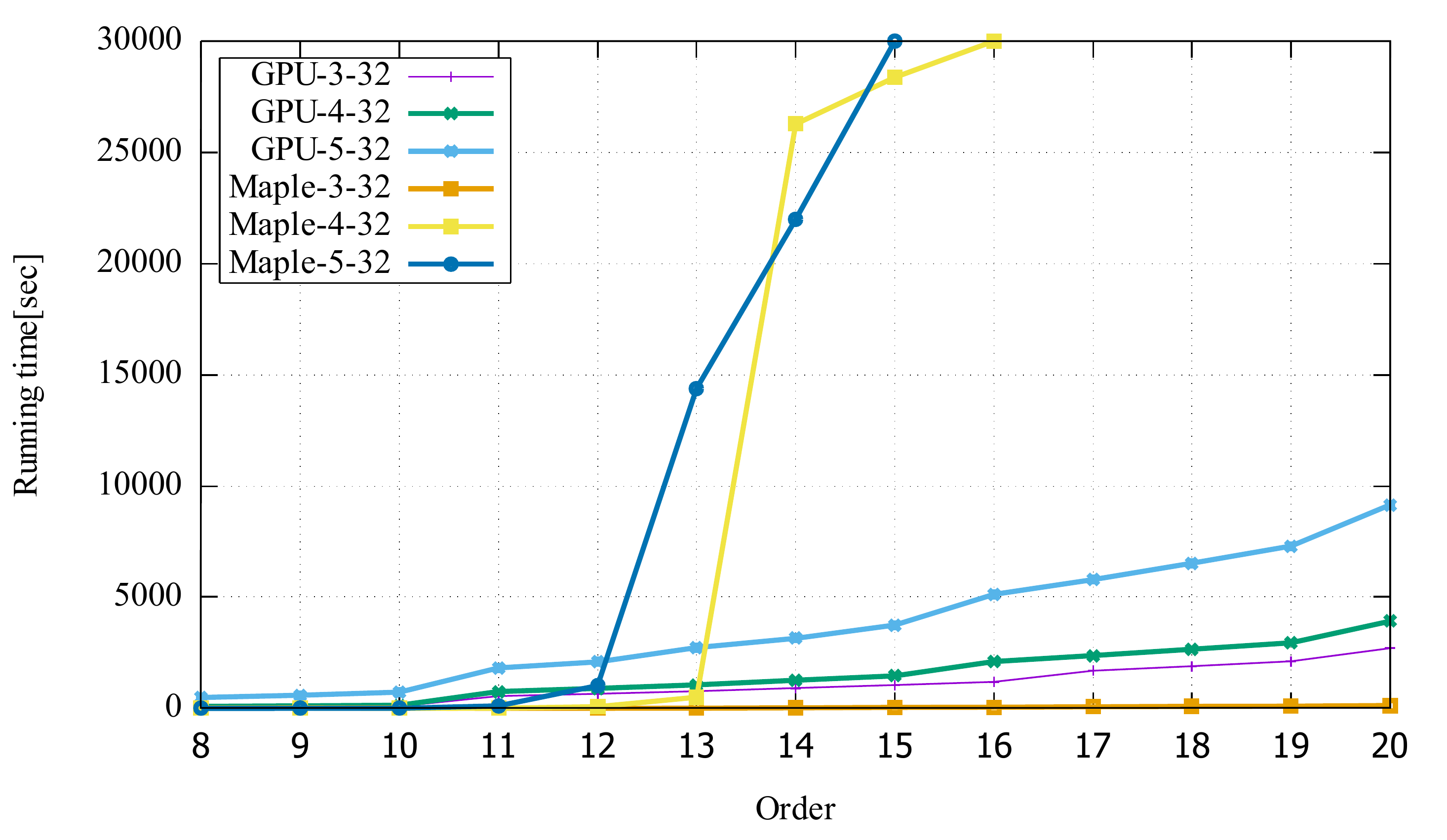}
\end{minipage}%
}%
\subfigure[Memory overhead on GPU versus Maple.]{
\begin{minipage}[t]{0.5\linewidth}
\label{memory_overhead_1}
\centering
\includegraphics[width=3.0in ,height= 2.0in]{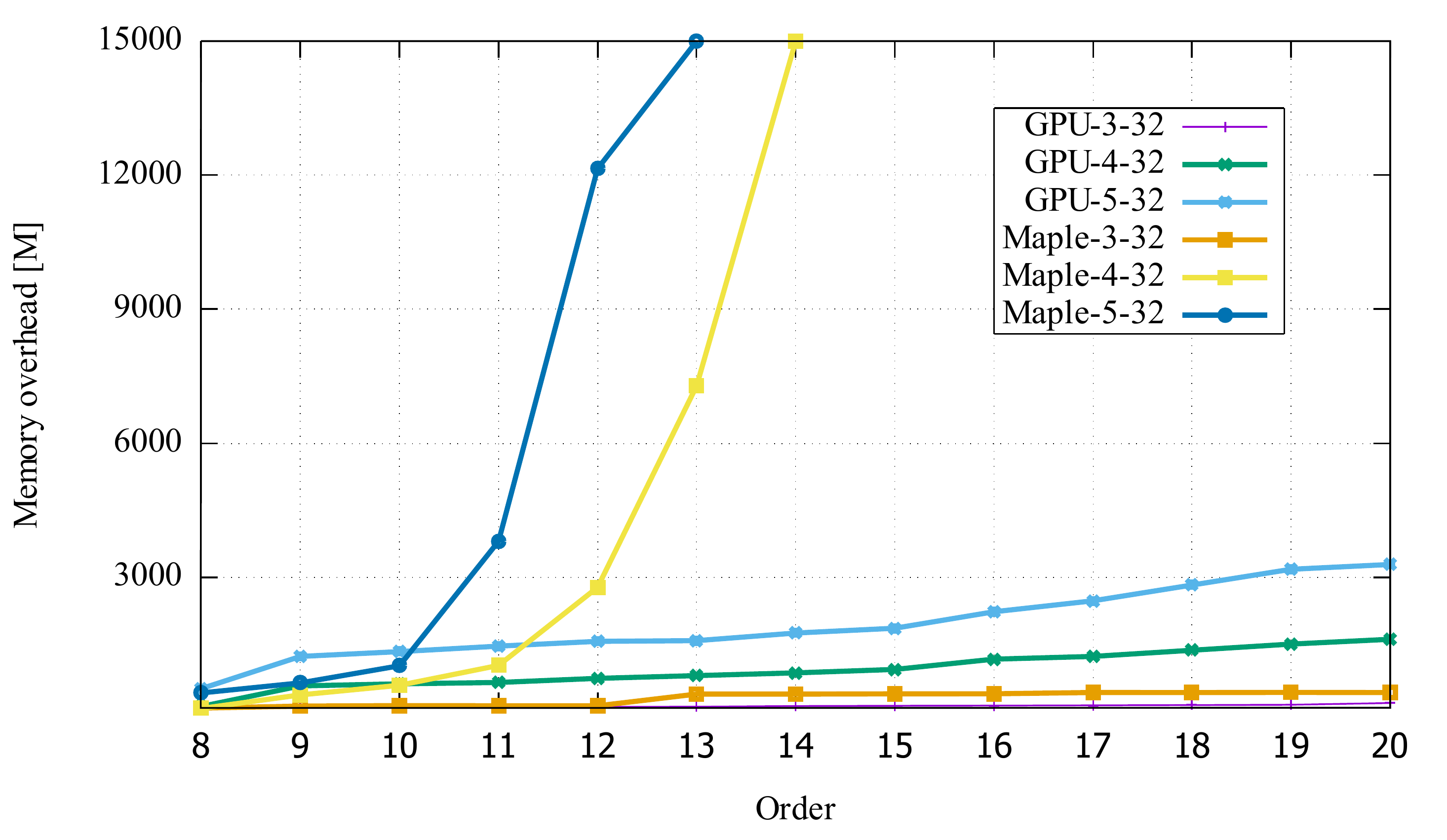}
\end{minipage}%
}%
\centering
\caption{Running time and memory overhead comparison. The max degree of each variate is assigned as 32 while the amount of variate is 3, 4, 5 for matrix whose order is from 8 to 20. Source: Various data are taken from Table \ref{tab:exp_data_1}. For example, "GPU-3-32" means matrix with 3 variates whose max degree are all 32 is running on GPU} 
\label{exp_1}
\end{figure*}

\begin{figure*}[htbp]
\centering
\subfigure[Running time on GPU versus Maple.]{
\begin{minipage}[t]{0.5\linewidth}
\label{running_time_2}
\centering
\includegraphics[width=3.0in , height= 2.0in]{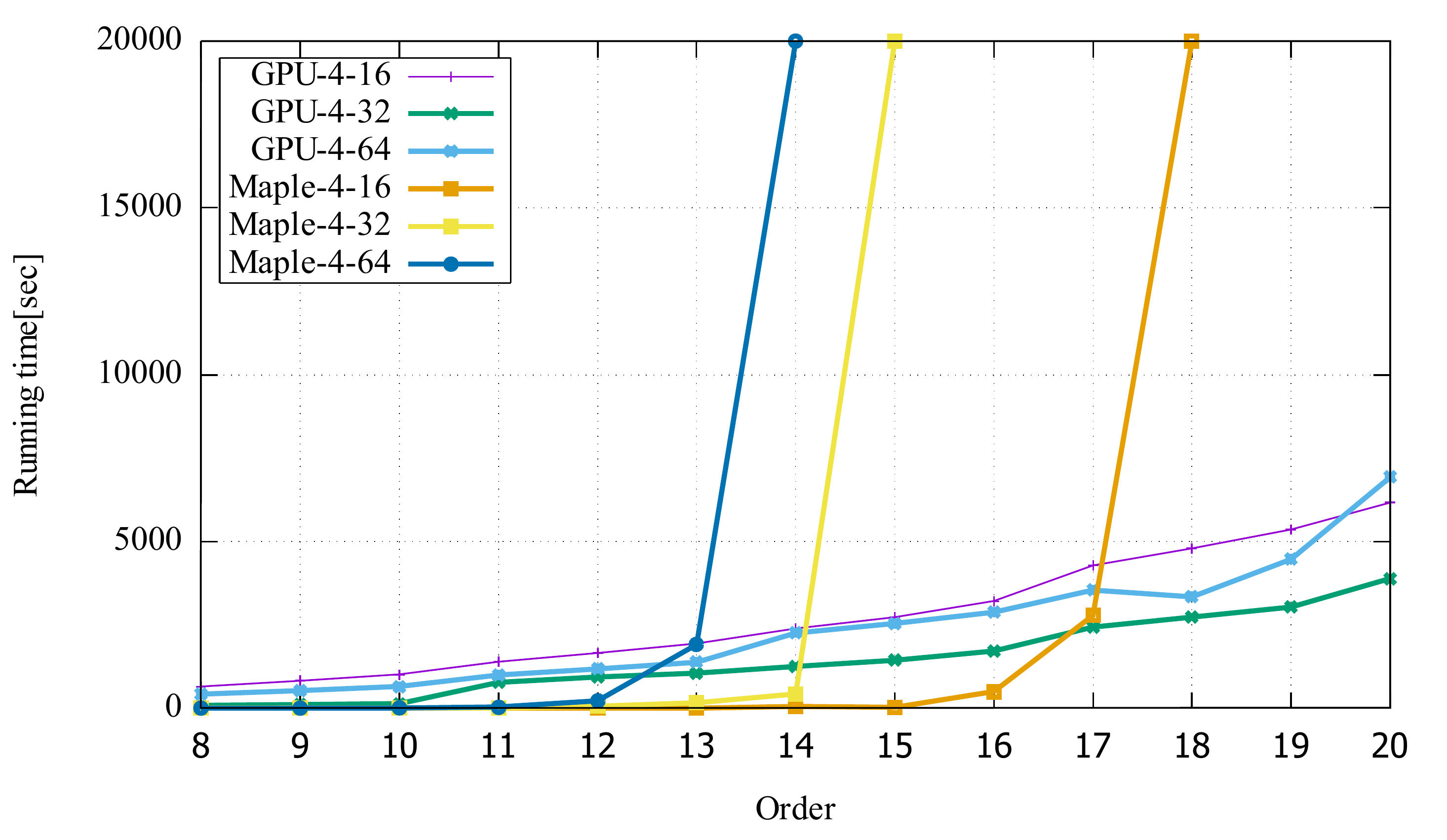}
\end{minipage}%
}%
\subfigure[Memory overhead on GPU versus Maple.]{
\begin{minipage}[t]{0.5\linewidth}
\label{memory_overhead_2}
\centering
\includegraphics[width=3.0in ,height= 2.0in]{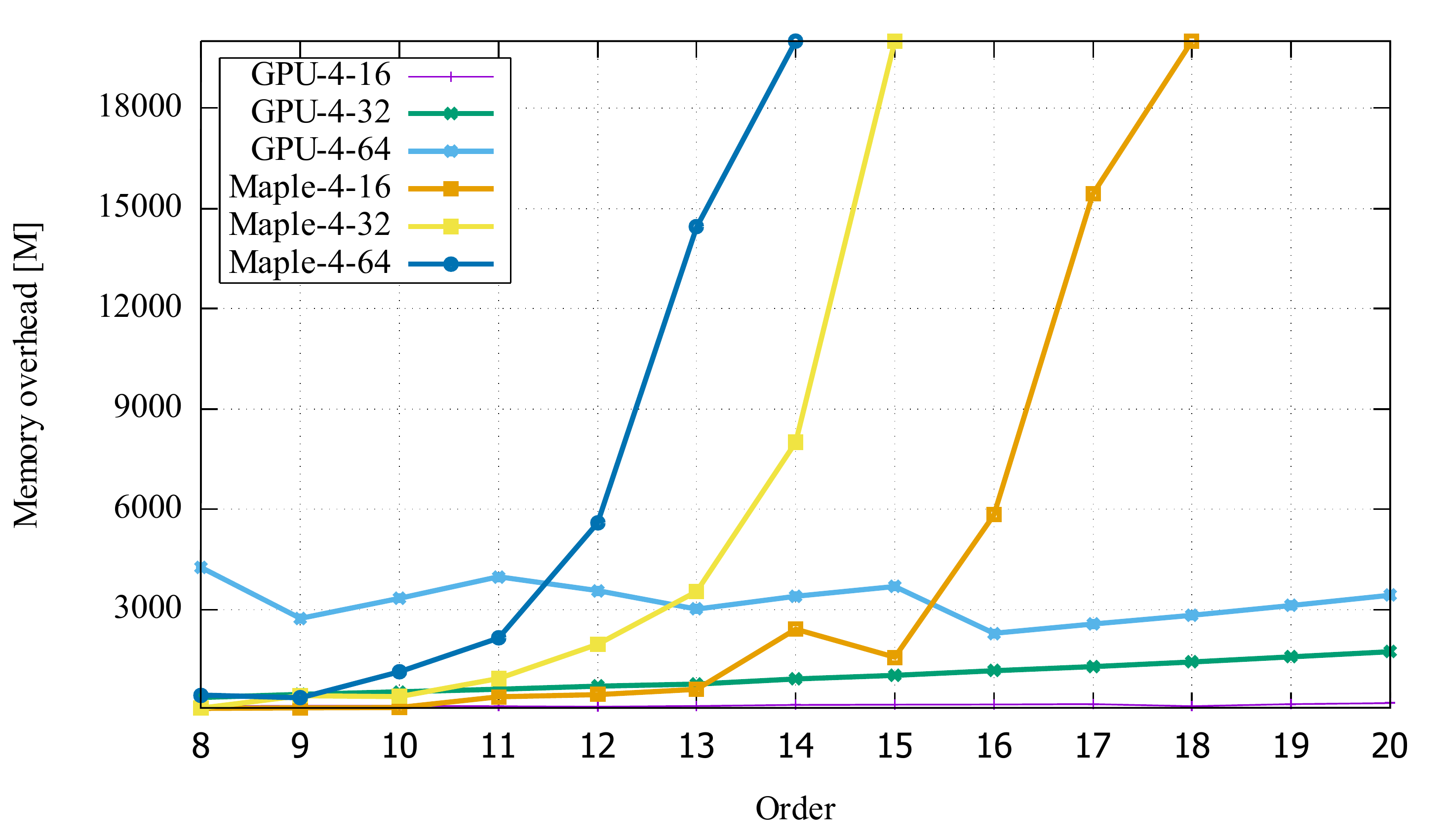}
\end{minipage}%
}%
\centering
\caption{Running time and memory overhead comparison. The account of variate is limited into 4 while max degree of each variate is 16, 32, 64 for matrices with same order. Source: Various data are taken from Table \ref{tab:exp_data_2}.} 
\label{exp_2}
\end{figure*}

\subsection{Running time and Memory}
We compare our implementation of multivariate polynomial determinant over big prime field for $order$ between 8 to 20 on GPU against a comparable approach based on a host-based determinant algorithm from 64-bit Maple 18, which provides a very efficient built-in determinant implementation by using Gaussian elimination. For polynomials of each matrix, we assign different amount of variables with various of degree as well as integer coefficients ranging from -100 to 100. Moreover, we need big primes that are over 9 bits in our GPU implementation and we also need at least two big primes because constructing prime field with single prime is meaningless for large scale computing.

The running time and memory overhead has been measured with two different multivariate polynomial matrix input groups. For matrix with different orders, we first set the max degree of each variate as 32 while the amount of variate is 3, 4 and 5. Experimental results are gathered in Table \ref{tab:exp_data_1} in Appendix A and shown in Figure \ref{exp_1}. We also have implemented and tested on GPU and Maple with inputs owning 4 variates whose max degree is limited as 16, 32, 64. Experimental results are gathered in Table \ref{tab:exp_data_2} and shown in Figure \ref{exp_2}.

Figure \ref{running_time_1} and Figure \ref{running_time_2} shows the running time of our GPU algorithm based on big prime field measured on NVIDIA Geforce 2080Ti compared to the Maple internal determinant method with different inputs. For both experimental results, our GPU approach goes through a process of stationary increment for all instances we have tried, being smooth and steady. This behavior is expected since, with increasing matrix degree, the increasing data triggers the rise on the number of thread blocks, which could be large enough to keep the device busy, leading to better hardware utilization. Besides, we can see that the optimized determinant method realizing on NVIDIA device results in impressive speedups especially when orders are over 14 or 15, which indicates that our approach would scale well on GPUs with large number of CUDA cores.

The original Maple determinant algorithm has outperformed our GPU method in some cases such as orders between 8 to 12 in Figure \ref{running_time_1} and 8 to 11 in Figure \ref{running_time_2}, which shows that our GPU approach may be less effective for lower order or lower degree matrix. However, some significant performance drops are noticeable when increasing the matrix size according to the test results on Maple. For instance, there are several jumps for orders between 12 to 15 in Figure \ref{running_time_1} and orders between 12 to 17 in Figure \ref{running_time_2}. This is because matrices with much higher orders or degree amplify the arithmetic intensity of computations, but Maple lacks the ability to afford it.

Figure \ref{memory_overhead_1} and Figure \ref{memory_overhead_2} shows the memory overhead during the process in computing polynomial determinant. Similar to experiments on running time above, Figure \ref{memory_overhead_1} summarizes the difference on memory with different amounts of variates and same degree compared to Figure \ref{memory_overhead_2} on memory with same amount of variates but different degree. 

Analyzing the trend of memory fluctuations in both situations, we have observed that Maple reaches hardware saturation much faster than our GPU method. For matrix whose order is over 15 or 16 approximately, we report the memory use and running time of Maple at the time when memory is exhausted, accompanying with the break off of the calculation. Meanwhile, decrease of the reported running time during these orders is reliable since expanding matrix size contributes to equivalent memory use in a much shorter time. That exposes the limitation of Maple. As it entails nearly 100\% or 200\% memory overhead which is presented in the $M\_Maple$ column in Table \ref{tab:exp_data_1} and Table \ref{tab:exp_data_2}, its main drawback is that it is not suitable for applications in which the memory resources are constrained.


On the contrary, our algorithm relies on thread-level parallelism and possesses a steady growing rate, which points out that our computation does not go back on as much memory overhead as that is for Maple, showing much more powerful computation capability. It can be noticed that using Maple for larger order or denser matrices would only result in memory soaring overall performance while our optimized GPU approach has memory overhead approximately less than 22\% during the experiment thanks to the on-chip memory for temporary storage. The reason why our GPU method is more advantageous than Maple is not only splitting the whole computation into 4 stages, which makes memory allocation transparent and manageable, but the modular method allows transformation from symbolic computation to numerical computation available, which is crucial for large-scale computing on GPU. Therefore, determinant for larger matrices that could not be obtained through Maple due to much more complex computation will be available under our GPU computation routine.


The sparsity and density is also crucial to performance when computing polynomial determinant. Recall that we extend each polynomial so that the degree of each variate reaches the one that is the first power of 2 larger than the original value, it is noticeable in Figure \ref{running_time_2} that the degree of each variate from matrices composing line "GPU-4-16" in Table \ref{tab:exp_data_2} is much closer to 16 while those from matrices composing "GPU-4-32" and "GPU-4-64" is far away from 32 and 64 respectively. In other words, "GPU-4-16" is much denser than "GPU-4-32" and "GPU-4-64". Thus, it explains why "GPU-4-16" costs more time even though it has lower degree than the others, which indicates that our parallel algorithm on GPU is also sensitive to sparsity and density of the matrix.


\begin{figure}
\includegraphics[width=3.2in, height=1.85in]{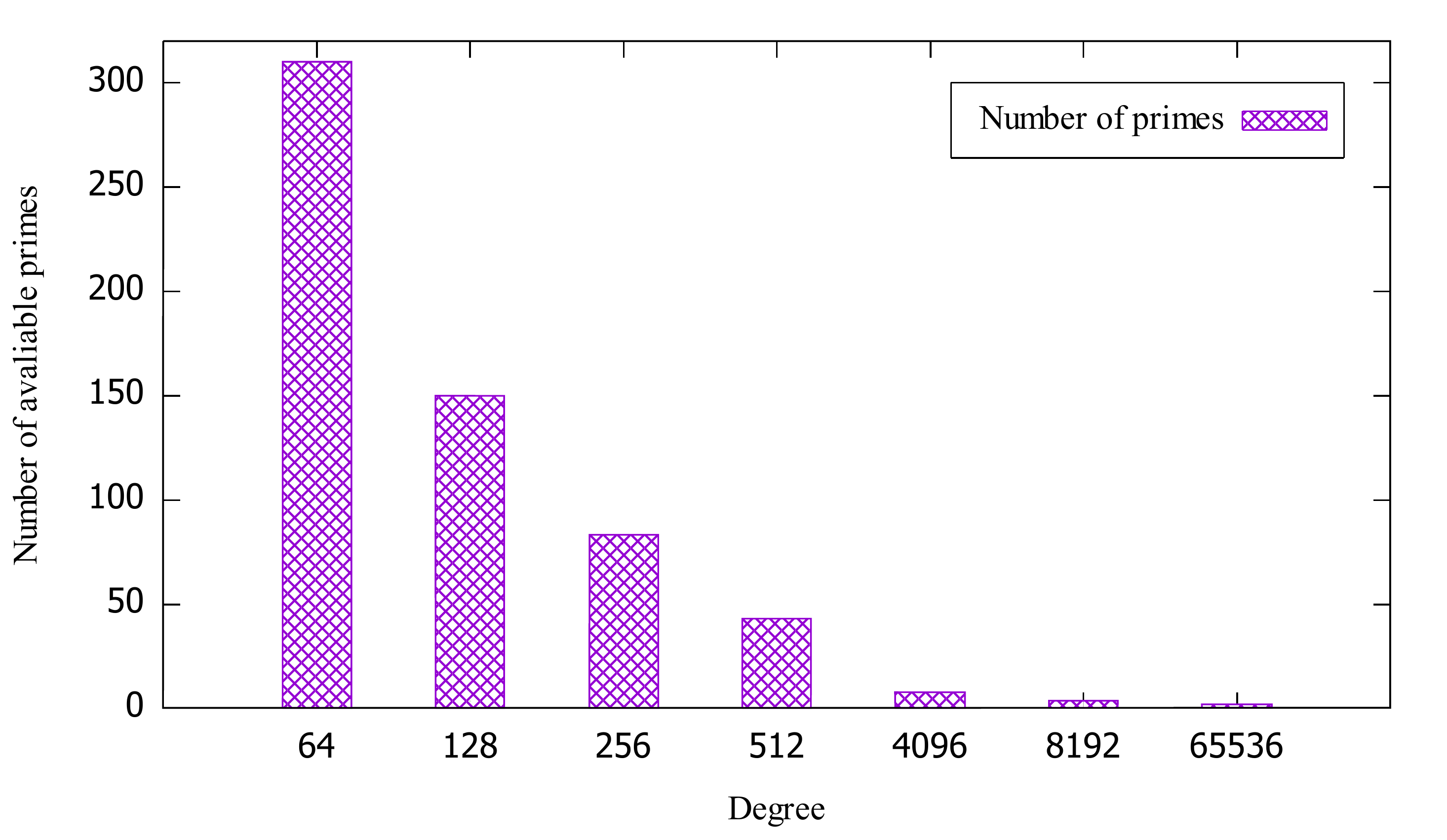}
\caption{The number of primes which own 64, 128, 256, 512, 4096, 8192 and 65536 primitive roots within the first 10000 primes over $1*10^9$.}
\label{available_primes}
\end{figure}

\begin{table}
\centering
\begin{tabular}{|c|c|c|}
\hline
Degree & Numbers of primes &  Acceleration\\
\hline
64 & 310  & 40 \\\hline
128 & 150  & 20 \\\hline
256 & 83   & 8 \\\hline
512 & 43  & 6\\\hline
4096 & 8  & 0.67 \\\hline
8192 & 4  & 0.5 \\\hline
65536 & 2  & 0.2\\\hline
\end{tabular}
\caption{The exact number of available primes we could use and the average increasement of available primes per 1000 primes.}
\label{available_primes_data}
\end{table}

\subsection{Primes we could use}
In \cite{DBLP:journals/ijfcs/MazaX11}, a balanced bivariate multiplication was proposed and held that the performance is optimal computing multi-dimensional FFT when multivariate and univariate multiplications are reduced to bivariate one whose partial degrees of the product are equal. The transformation from multivariate multiplication to bivariate one may trigger the very awareness of its performance, the problem, however, is that the degree of each of bivariate would be extremely high, a not unusual occurrence with polynomial multiplication. The observation also suggests that it is hard to find qualified primes possessing enough primitive roots of the unity, instead, the higher the degree of the variate is, the larger the primes is, and thus, the fewer the number of available primes would be. 

To confirm these assumptions, we count the number of available primes that has 64, 128, 256, 512, 4096, 8192 and 65536 primitive roots within the first 10000 primes over $1*10^9$. According to the data from Figure \ref{available_primes} and Table \ref{available_primes_data}, we could conclude that primes are abundance when the variate owns relatively lower degree such as 64 and 128, and the very paucity of available primes for variate with higher degree such as 8192 and 65536 which constructs a stark contrast compared to the former one and the complexity of the process detecting them make it difficult for FFT to process multiplication. Although we just consider the first 10000 primes over $1*10^9$, it is convinced that more primes would be collected by variates with lower degree while fewer one by its counterpart. 

Therefore, a balanced bivariate method for FFT may encounter potential difficulty that it will spend more time finding sufficient primes constructing prime field. For instance, when a 4-variates polynomial whose degree is 256 respectively is contracted into bivariate one whose degree is $256*256$, that is 65536, without losing value, only 2 possible primes of first 10000 could be used for prime field while the possibilities could be chosen is 310 when the degree of one of the variates is 64 and 150 when the degree is 128. In addition, it is much easier for bivariate with high degree to use primes over 9 digits, which is unfriendly and may compute slower on GPU. Multivariate FFT computation for our GPU method would be processed for several times, yet it is much technically feasible.

\subsection{Robustness of GPU algorithm}
Thanks to our 4-stage parallel approach, the input of the post-stage depends on the results of its precursor. Besides, we compute FFT on GPU for each element of multivariate polynomial matrix one after another and the results of other stages could also be split into several parts. Therefore, we could conclude that our GPU algorithm could recover its computation at any point even though it is interrupted by some other factors. It is advantageous and even reliable especially when computation of polynomial matrix may last several hours or the input matrix is much more complicated.  

Altogether, our results present a significant performance improvement over a host-based implementation. Furthermore, we have shown that our approach is well-suited for realization on massively-threaded architectures possessing a great scalability potential because of the large degree of parallelism inherent to the modular approach. In addition, calculation could be recovered and go on without losing preciseness when it is interrupted, which demonstrates the robustness and applicability for the proposed method.

\section{Prediction of running time}
In this section, we outline the practical way how to predict running time computing multivariate polynomial determinant on GPU. Admittedly, libraries for doing polynomial system or number theory such as Maple and NTL \footnote{http://www.shoup.net/ntl} are able to compute determinant, but they do not provide a friendly environment. When processing matrix, especially for polynomial matrix with high order and degree based on libraries above, what we can do is nothing but waiting for the answer since it is such a encapsulated entity that we do not know which step it is processing at all. Besides, some of libraries will not show an exception message at the time when functions in these libraries could not afford the input polynomial matrix, which is less effective and consumes huge amount of time.

Our algorithm provides a more humanitarian approach towards the fore mentioned circumstances. Recall all stages above, we transfer the computation, which is concerned as a black box, into the one which could be considered as a white box where each stage becomes transparent. 


Experimental resullts are gather in Table \ref{tab:exp_data_1} and Table \ref{tab:exp_data_2} where column 4-7 shows relative contribution of different stages of our GPU approach to the overall time. It is understandable that the time for computing FFT is dominating since this is a key part of the algorithm. We could also conclude that the running time of polynomial determinant mostly depends on the size, degree, and density and sparsity of the chosen matrix according to the analysis above. Therefore, we could predict the running time according to the formula shown in Equation \ref{predictRunningTime} where $C_{p}$ and $r$ are the amount of prime we use and the order of the polynomial matrix respectively. $T_{e_i}$ for $0 \leq i \le r^2-1$ represents the time we use to compute FFT over one of the big prime field for each element of polynomial matrix, which is still a multivariate polynomial accurately. Sometimes the situation is that matrix is constructed by repetitious factors, a not unusual occurrence with mathematic applications such as Sylvester matrix. Consequently, it is necessary to set $\mu = \frac{k}{r^2}$ as replication factor where $k$ means the number of element that is unique in the matrix. We omit time consuming on the rest part of the algorithm since DET, IFFT and CRT weighs little.

\begin{equation}
T = C_{p}r^2 \cdot Round( \frac{\sum_{i=0}^{e_{r^2-1}}T_{e_i}}{r^2},2) \cdot \mu
\label{predictRunningTime}
\end{equation}

For instance, assuming that the input $16 \times 16$polynomial matrix has four variables whose degrees are $(30, 30, 27, 28)$ and no duplicated elements, that the $\mu$ is 1 and the average time for $T_{e_i}$ consumes about 1.36s over finite field composed by six large primes suggests we could predict the running time. The predicted time is 2088.96s, which is approximate to its real time 2107.39s, while Maple costs more than 17 hours. 

Understandably, prediction during the process is essencial especially for large scale scientific computing where computation may last for several days. It provides an approximate reference range which is instructive when judging whether or not we should continue or stop the process. It is also time-saving in verifying correctness since we could definitely drop the procedure once we find $T_{e_0}$ or $T_{e_1}$ is not equivalent to what we expect.

\section{Open problem we solve}
One of the open problems is solving harmonic elimination equations for a multilevel converter. For a three-phase system with 4 DC sources, the mathematical statement could be defined as 

\begin{equation}
  \left\{
   \begin{array}{l}
   \cos(\theta_1) - \cos(\theta_2) +	\cos(\theta_3) - 	\cos(\theta_4) = m  \\
   \cos(5\theta_1) - \cos(5\theta_2) +	\cos(5\theta_3) - \cos(5\theta_4) = 0  \\
   \cos(7\theta_1) - \cos(7\theta_2) +	\cos(7\theta_3)- \cos(7\theta_4) = 0  \\
   \cos(11\theta_1) - \cos(11\theta_2) + \cos(11\theta_3) - \cos(11\theta_4) = 0  \\
   \end{array}
  \right.
  \end{equation}
where $\theta_k$ are switching angles  and m is constant.

The resulting equation system, characterizing the harmonic content, is nonlinear, contain trigonometric terms and are transcendental in nature. Previous work in \cite{Bouhali2013SolvingHE, AUSHEM} has shown that these equations can be converted into polynomial equations, that is converting trigonometric elements of each equation to polynomial elements, which are then solved using the method of resultants from elimination theory. However, the degrees of the polynomials are large if there are several DC sources. As a result, contemporary algebra computing software tools such as Maple and Mathematica reach their computational limitations in solving the system of polynomial equations when computing the resultant polynomial of the system. Although in \cite{anasheefmc}, the theory of symmetric polynomials is exploited to reduce the degree of the polynomial equation, the computation using contemporary computer algebra software tools still appear to reach their limit when one goes to five or more DC sources. This computational complexity is because the degrees of the polynomials are large which in turn requires the symbolic computation of the determinant of large $n \times n$ matrices. 

In our GPU method, however, we skirt this hardest issue, concentrating on areas of numerical computing itself. Here we focus on a three-phase system with 9 DC sources. Based on the theory of symmetric polynomials, the original harmonic elimination equations could be rewritten into polynomial equations $(g_1, g_2, g_3, g_4, g_5, g_6, g_7, g_8, g_9)$ in terms of $(d_1, d_2, d_3, d_4, d_5,d_6, d_7, d_8, $ $ d_9)$. Then, a systematic procedure is introduced to eliminate $d_i$ and uses the notion of $resultant$. 

The $resultant$ polynomials eliminating $d_9$ could be written in the form:


\begin{center}
$r_1(d_1, \ldots, d_8) = Res \big(g4(d_1, \ldots, d_9), g5(d_1, \ldots, d_9), d_9 \big) $

$r_2(d_1, \ldots, d_8) = Res \big(g4(d_1, \ldots, d_9), g6(d_1, \ldots, d_9), d_9 \big)$

$\vdots$

$ r_5(d_1, \ldots, d_8) = Res \big(g4(d_1, \ldots, d_9), g9(d_1, \ldots, d_9), d_9 \big)$.

\end{center}


When eliminating $d_5$, Maple and Mathematica are not able to afford the computational burden. Thus, we consider the optimized GPU method to solve this resultant. Since the determinant of Sylvester matrix is resultant, that is,
\begin{equation}
 Res \big(g5(d_1, \ldots, d_5), g6(d_1, \ldots, d_5), d_5 \big)  \triangleq det S_{g5, g6}(d_1, \ldots, d_4),
\end{equation}
we attempt to construct Sylvester matrix $S$ for $d_5$. The question now is "Given an $15 \times 15$ symbolic matrix with variables $(d_1, d_2, d_3, d_4)$ whose degrees are (298, 171, 119, 45) respectively, how to solve its determinant?".

Obviously, the determinant could be easily computed on GPU as well as get exact solution without losing accuracy. A small trick here is that since there are several duplicate terms when constructing Sylvester matrix, the FFT result over one prime could be generated just by computing FFT for each item of $g_5$ once and that of $g_6$ once instead of visiting each element of the Sylvester matrix iteratively. Finally, we take roughly 900 minutes to get the resultant which is eliminating $d_5$. 

The remaining solution for eliminating $d_i$ could be done in fashion similar to the way for $d_5$ when confronting issues that overload on CPU. Then, the procedure is to substitute the solutions of $d_i$ into $g_i$ and solve for the roots $(d_1, d_2, \ldots, d_9)$ and thus, all possible solutions for switching angles $\theta_k$ are solved.


\section{Conclusions}
We have presented the design and implementation of the determinant of multivariate polynomial matrices for modern GPUs. We have enhanced both the performance of the traditional determinant algorithm as well as overall staged approach. Combined with insights that lead to greatly improved performance of elementary symbolic computation, a new parallel approach that is efficient for GPUs is presented and we have shown that this is much faster than traditional determinant computation on CPU. In addition, our algorithm owns less memory overhead as well as smooth memory increase. We have also observed that there is no accuracy loss during our process and the procedure could be continued at any point. Though some questions can be large and complex, we have provided a time prediction that helps user to estimate approximate running time. Finally, we have solved an open symbolic problem relating to harmonic elimination equations. As a future work, we plan to extend our work to multi-GPU environments. We believe that our optimized parallel determinant approach can be used as a building block for a multi-GPU version.

\begin{acks}
To Robert, for the bagels and explaining CMYK and color spaces.
\end{acks}

\bibliographystyle{ACM-Reference-Format}
\bibliography{sample-base}

\appendix

\appendix
\section{Appendix A}
We provide detailed results for experiments above, including accurate running time and memory overhead for each steps of our GPU-based parallel method and those on Maple.

Table \ref{tab:exp_data_1} and Table \ref{tab:exp_data_2} gathers running times for computing multivariable polynomial determinant both on GPU and CPU measured by second as well as memory overhead measured by megabytes with different inputs. In GPU column, we account for all stages of our optimized algorithm including recovering the large integer coefficients and presenting the result on CPU. We have also listed time use of each stage in detail. Similarly, memory overhead is reported in Table \ref{tab:exp_data_1} and Table \ref{tab:exp_data_2} as  M\_GPU and M\_Maple for our GPU implementation and Maple respectively.

\begin{table*}
  \caption{Timing and memory use of our optimized GPU algorithm and Maple. The max degree of each variate is assigned as 32 while the account of variate is 3, 4, 5 for matrix whose order is from 8 to 20. Time that each stage expends on GPU has also been recorded. }
  \label{tab:exp_data_1}
\setlength{\tabcolsep}{1.1mm}{
  \begin{tabular}{ccccccccccc}
    \toprule
Rank&Prime&Degree&FFT(s)&DET(s)&IFFT(s)&CRT(s)&GPU(s)&Maple(s)&M\_GPU(M)&M\_Maple(M)\\
\midrule
\multirow{3}*{8}&3 & 16, 16, 18 & 62.16 & 0.14 & 0.97 & 0 & 63.27 & 0.16 & 93.696 & 72.3 \\

&3 & 17, 17, 18, 16 & 86.66 & 4.26 & 1.43 & 0.03 & 92.38 & 0.2 & 109.056 & 64.2 \\
&3 & 17, 17, 18, 16, 18 & 313.87 & 170.99 & 7.69 & 1.16 & 493.71 & 0.92 & 506.88 & 410.5 \\
\hline

\multirow{3}*{9}&3 & 20, 18, 22 & 78.57 & 0.18 & 0.96 & 0 & 79.71 & 0.57 & 95.232 & 116.8 \\

&3 & 17, 17, 18, 21 & 109.62 & 5.30 & 1.43 & 0.03 & 116.38 & 1.18 & 568.32 & 367 \\

&3 & 17, 17, 18, 21, 21 & 375.36 & 210.56 & 7.57 & 1.13 & 594.62 & 4.63 & 1228.8 & 636.3 \\
\hline

\multirow{3}*{10}&3 & 22, 22, 24 & 96.94 & 0.20 & 0.97 & 0.01 & 98.12 & 1.05 & 96.768 & 128.3 \\

&3 & 22, 22, 24, 25 & 135.36 & 6.41 & 1.43 & 0.03 & 143.23 & 3.01 & 614.4 & 581.3 \\

&3 & 22, 20, 24, 25, 22 & 469.73 & 252.94 & 7.49 & 1.14 & 731.3 & 9.19 & 1336.32 & 1020 \\
\hline

\multirow{3}*{11}&4 & 26, 23, 26 & 550.66 & 0.32 & 4.54 & 0.01 & 555.53 & 3.86 & 99.072 & 130.3 \\

&4 & 21, 20, 21, 22 & 744.01 & 99.60 & 6.24 & 50 & 760.26 & 10.69 & 645.12 & 1027.7 \\

&4 & 21, 20, 21, 22, 22 & 1413.34 & 389.91 & 15.66 & 1.47 & 1820.38 & 116.26 & 1459.2 & 3806.2 \\
\hline

\multirow{3}*{12}&4 & 24, 20, 24 & 655.24 & 0.36 & 4.54 & 0 & 660.14 & 6.93 & 99.84 & 128.8 \\

&4 & 24, 20, 26, 21 & 885.41 & 11.47 & 6.24 & 0.05 & 903.17 & 76.83 & 737.28 & 2773.3 \\

&4 & 24, 20, 26, 21, 22 & 1623.35 & 453.98 & 14.91 & 1.20 & 2093.44 & 1023 & 1566.72 & 12145.3 \\
\hline

\multirow{3}*{13}&4 & 26, 23, 23 & 768.94 & 0.41 & 4.56 & 0 & 773.91 & 12.91 & 102.912 & 381.8 \\

&4 & 29, 29, 29, 31 & 1038.99 & 13.10 & 6.25 & 0.04 & 1058.38 & 505.4 & 798.72 & 7284.9 \\

&5 & 25, 25, 23, 25, 25 & 2063.03 & 653.97 & 16.99 & 1.70 & 2735.69 & >14391.93 & 1582.08 & >26998.3 \\
\hline

\multirow{3}*{14}&4 & 27, 29, 27 & 910.87 & 0.58 & 4.66 & 0 & 916.11 & 25.78 & 118.272 & 383.4 \\

&4 & 27, 26, 27, 26 & 1239.62 & 20.50 & 6.43 & 0.06 & 1266.61 & 14726.44 & 860.16 & 21290.8 \\

&5 & 27, 26, 27, 26, 27 & 2395.9 & 736.66 & 17.18 & 1.75 & 3151.49 & >21740.4 & 1751.04 & >25368 \\
\hline

\multirow{3}*{15}&5 & 29, 28, 25 & 1045.8 & 0.65 & 4.65 & 0 & 1051.1 & 47.92 & 122.88 & 389.8 \\

&5 & 29, 29, 29, 28 & 1432.79 & 23.95 & 6.42 & 0.06 & 1463.22 & >46272.76 & 936.96 & >23771 \\

&5 & 29, 29, 29, 28, 25 & 2880.59 & 849.20 & 17.70 & 1.85 & 3749.43 & >21987.94 & 1858.56 & >23899.4 \\
\hline

\multirow{3}*{16}&5 & 30, 30, 27 & 1189.72 & 0.72 & 4.64 & 0.1 & 1195.09 & 55.62 & 127.488 & 392.6 \\

&6 & 30, 30, 27, 28 & 2064.81 & 34.44 & 8.07 & 0.07 & 2107.39 & >25668 & 1167.36 & >25668.0 \\

&6 & 30, 30, 27, 28 & 4053.78 & 1049.68 & 21.69 & 2.20 & 5127.35 & >14274.35 & 2227.2 & >22739.5 \\
\hline

\multirow{3}*{17}&6 & 29, 31, 28, 29 & 1689.43 & 1.00 & 5.86 & 0 & 1696.29 & 73.05 & 133.632 & 421.6 \\

&6 & 29, 31, 28, 29 & 2331.35 & 37.77 & 8.06 & 0.08 & 2377.26 & >49595.02 & 1228.8 & >27189.2 \\

&6 & 29, 31, 28, 29, 27 & 4549.6 & 1214.64 & 21.67 & 2.22 & 5788.13 & >29375.2 & 2472.96 & >26968.6 \\
\hline

\multirow{3}*{18}&6 & 30, 30, 31 & 1893.41 & 1.10 & 5.85 & 0.01 & 1900.37 & 93.59 & 138.24 & 419.3 \\

&6 & 30, 31, 31, 29 & 2612.53 & 42.69 & 8.06 & 0.09 & 2663.37 & >12763.78 & 1367.04 & >20400.1 \\

&6 & 30, 31, 31, 29, 31 & 5109.1 & 1395.90 & 21.89 & 2.22 & 6529.11 & >26375.2 & 2826.24 & >21968.6 \\
\hline

\multirow{3}*{19}&6 & 27, 31, 30  & 2110.66 & 1.22 & 5.86 & 0.01 & 2117.75 & 98.68 & 145.92 & 422.2 \\

&6 & 31, 31, 30, 30 & 2890 & 47.54 & 8.07 & 0.07 & 2950 & >27425.9 & 1505.28 & >23978.1 \\

&6 & 31, 31, 30, 30, 27 & 5711.13 & 1564.12 & 21.82 & 2.17 & 7299.24 & >25109.9 & 3179.52 & >22127.6 \\
\hline

\multirow{3}*{20}&7 & 30, 29, 30 & 2700 & 1.55 & 6.75 & 0 & 2710 & 131.91 & 184.32 & 419.2 \\

&7 & 30, 31, 30, 31 & 3850 & 60.82 & 9.68 & 0.10 & 3920 & >20475.3 & 1612.8 & >25968.3 \\

&7 & 30, 31, 30, 31, 30 & 7224.69 & 1899.80 & 24.96 & 2.49 & 9151.94 & >24445.2 & 3287.04 & >27938.2 \\

  \bottomrule
\end{tabular}}
\end{table*}

\begin{table*}
  \caption{Running time and memory overhead of our optimized GPU algorithm and those on  Maple. The account of variate is limited into 4 while max degree of each variate is 16, 32, 64 for matrices with same order. The range of matrix order is from 8 to 20 and time that each stage expends on GPU has also been recorded. }
  \label{tab:exp_data_2}
\setlength{\tabcolsep}{1.2mm}{
  \begin{tabular}{ccccccccccc}
    \toprule
Rank&Prime&Degree&FFT(s)&DET(s)&IFFT(s)&CRT(s)&GPU(s)&Maple(s)&M\_GPU(M)&M\_Maple(M)\\
\midrule
\multirow{3}*{8}& 3 & 14, 13, 14, 15 & 647.47 & 0.27 & 10.1 & 0 & 657.84 & 0.06 & 107.52 & 40.3\\

& 3 & 16, 16, 18, 24 & 86.59 & 4.34 & 1.43 & 0.03 & 92.39 & 0.14 & 368.64 & 51.1\\

& 3 & 33, 35, 32, 48 & 334.93 & 87.28 & 6.79 & 0.6 & 429.6 & 0.99 & 4254.72 & 439.7\\
\hline

\multirow{3}*{9}& 3 & 14, 15, 15, 15 & 819.01 & 0.33 & 10.11 & 0.01 & 829.46 & 0.27 & 122.88 & 58.2\\

& 3 & 20, 18, 22, 24 & 109.7 & 5.15 & 1.42 & 0.03 & 116.3 & 0.85 & 460.8 & 419.8\\

& 3 & 36, 32, 38, 38 & 427.65 & 100.58 & 6.42 & 0.55 & 535.2 & 1 & 2734.08 & 358.6\\
\hline

\multirow{3}*{10}& 3 & 14, 15, 15, 15 & 1011.78 & 0.39 & 10.11 & 0 & 1022.28 & 0.54 & 122.88 & 67.8\\

& 3 & 22, 22, 24, 26 & 135.36 & 6.27 & 1.44 & 0.04 & 143.11 & 0.97 & 537.6 & 394.8\\

& 3 & 39, 35, 39, 46 & 529.31 & 122.56 & 6.49 & 0.55 & 658.91 & 11.72 & 3333.12 & 1136.7\\
\hline

\multirow{3}*{11}& 4 & 11, 13, 11, 15 & 1386.37 & 0.64 & 11.46 & 0 & 1398.47 & 1.49 & 103.68 & 387.8\\

& 4 & 26, 23, 26, 27 & 764.71 & 9.99 & 6.43 & 0.04 & 781.17 & 10.03 & 614.4 & 934.4\\

& 4 & 36, 36, 35, 44 & 803.46 & 188.84 & 8.22 & 0.71 & 1001.23 & 43.59 & 3978.24 & 2144.2\\
\hline

\multirow{3}*{12}& 4 & 12, 15, 11, 15 & 1649.93 & 0.77 & 11.45 & 0 & 1662.15 & 1.52 & 92.16 & 451.8\\

& 4 & 24, 20, 24, 27 & 921.98 & 13.58 & 6.42 & 0.06 & 942.04 & 65.91 & 706.56 & 1958.2\\

& 4 & 44, 45, 42, 45 & 957.77 & 218.06 & 8.24 & 0.73 & 1184.8 & 229.37 & 3563.52 & 5585.6\\
\hline

\multirow{3}*{13}& 4 & 14, 13, 13, 14 & 1935.44 & 0.88 & 11.43 & 0.02 & 1947.77 & 7.13 & 112.128 & 611.8\\

& 4 & 26, 23, 23, 28 & 1039.11 & 12.85 & 6.23 & 0.04 & 1058.23 & 168.67 & 768 & 3536.4\\

& 4 & 50, 34, 41, 50 & 1124.3 & 250.45 & 8.25 & 0.71 & 1383.61 & 1916.7 & 3010.56 & 14437.9\\
\hline

\multirow{3}*{14}& 5 & 14, 14, 14, 15 & 2382.02 & 1.19 & 12.15 & 0.01 & 2395.37 & 58.68 & 144.384 & 2410.6\\

& 5 & 27, 26, 27, 27 & 1233.74 & 18.42 & 6.34 & 0.05 & 1258.64 & 434.51 & 921.6 & 8003.03\\

& 5 & 47,41,41,47 & 1889.29 & 357.85 & 11.65 & 0.9 & 2259.69 & >17595.85 & 3394.56 & >22978.5\\
\hline

\multirow{3}*{15}& 5 & 15, 15, 15, 15 & 2734.05 & 1.34 & 12.15 & 0.01 & 2734.05 & 34.63 & 153.6 & 1571.8\\

& 5 & 29, 28, 25, 26 & 1416.4 & 20.72 & 6.43 & 0.06 & 1443.61 & >23145.1 & 1029.12 & >21202.9\\

& 5 & 43, 48, 48, 48 & 2149.16 & 385.62 & 11.63 & 0.9 & 2547.31 & >34072.3 & 3686.4 & >25829.3\\
\hline

\multirow{3}*{16}& 5 & 15, 15, 15, 15 & 3206.34 & 1.54 & 12.55 & 0.01 & 3220.44 & 507.09 & 158.208 & 5827.8\\

& 5 & 30, 30, 27, 30 & 1686.13 & 28.98 & 6.64 & 0.05 & 1721.8 & >27517.9 & 1167.36 & >25845.3\\

& 5 & 50, 56, 48, 50 & 2438.66 & 433.64 & 11.37 & 0.89 & 2884.56 & >25982 & 2288.64 & >26801.4\\
\hline

\multirow{3}*{17}& 6 & 15, 15, 15, 15 & 4268.56 & 2.06 & 14.76 & 0 & 4285.38 & 2787.38 & 168.96 & 15432.7\\

& 6 & 30, 30, 28, 29 & 2389.29 & 38.83 & 8.37 & 0.08 & 2436.57 & >24308.9 & 1290.24 & >26754.1\\

& 6 & 51, 56, 47, 51 & 2945.2 & 587.67 & 12.4 & 1.03 & 3546.3 & >24909.6 & 2565.12 & >25432.8\\
\hline

\multirow{3}*{18}& 6 & 15, 15, 15, 15 & 4779.79 & 2.23 & 14.8 & 0 & 4796.82 & >13652.82 & 107.52 & >24903.3\\

& 6 & 30, 30, 31, 31 & 2686.81 & 42.45 & 8.38 & 0.07 & 2737.71 & >23571.9 & 1428.48 & >25494.3\\

& 6 & 49, 47, 51, 54 & 3344.29 & 653.45 & 12.76 & 1.07 & 3344.29 & >22541.8 & 2826.24 & >25767.5\\
\hline

\multirow{3}*{19}& 6 & 15, 15, 15, 15 & 5344.74 & 2.42 & 14.82 & 0 & 5361.98 & >8786.04 & 168.96 & >26496.6\\

& 6 & 27, 31, 30, 27 & 2983 & 46.87 & 8.39 & 0.07 & 3038.33 & >22409.8 & 1582.08 & >24309.7\\

& 6 & 49, 52, 50, 59 & 3733.73 & 723.41 & 12.84 & 1.08 & 4471.06 & >21090.8 & 3118.08 & >23907.6\\
\hline

\multirow{3}*{20}& 7 & 15, 15, 15, 15 & 6151.13 & 3.1 & 15.41 & 0.01 & 6169.65 & >21009.1 & 199.68 & >23407.9\\

& 7 & 30, 29, 30, 29 & 3825.02 & 59.07 & 9.73 & 0.09 & 3893.91 & >23094.8 & 1735.68 & >24090.7\\

& 7 & 52, 50, 52, 51 & 5988.57 & 920.55 & 17.84 & 1.28 & 6928.24 & >22804.3 & 3425.28 & >25702.6\\

  \bottomrule
\end{tabular}}
\end{table*}

\end{document}